\documentclass{amsart}
\usepackage[a4paper, top=3.0 cm,bottom=3.5 cm,left=3.0 cm,right=2.5 cm]{geometry}
\usepackage[all]{xy}
\usepackage{amsfonts}
\usepackage{amsmath}
\usepackage{amssymb}
\usepackage{mathrsfs}
\usepackage{latexsym}
\usepackage{amsthm}
\usepackage{xy}
\usepackage{enumitem}
\usepackage{xcolor}
\usepackage{setspace}
\newlist{legal}{enumerate}{10}
\setlist[legal]{label*=\arabic*.}
\usepackage{graphicx}

\usepackage{tikz-cd}
\usepackage{tikz} 

\usepackage[english]{babel}
\usepackage[latin1]{inputenc}
\usepackage[active]{srcltx}

\usepackage{amssymb,epsfig,amsxtra, amsmath}
\usepackage[mathscr]{eucal}
\usepackage{a4}
\input{xy}
\xyoption{all}
\newtheorem{defin}{Definition}[section]

\newtheorem{thm}{\textbf{Theorem}}[section]
\newtheorem{prova}{\textbf{Proposition/Example}}[section]

\newtheorem{lemma}[defin]{\textbf{Lemma}}

\newtheorem{prop}[defin]{\textbf{Proposition}}
\theoremstyle{definition}
\newtheorem{rem}[defin]{\textbf{Remark}}
\newtheorem{rems}[defin]{\textbf{Remarks}}
\newtheorem{exa}[defin]{\textbf{Example}}

\usepackage{hyperref}

\usepackage{tikz} 

\begin{document}
\title[Surfaces with Prym-canonical hyperplane sections]
{Surfaces with Prym-canonical hyperplane sections}

\author{MARTINA ANELLI}



\thanks{}



\date{}

\dedicatory{}

\commby{}


\begin{abstract}
In this paper, we will explicit some general properties regarding surfaces with Prym-canonical hyperplane sections and the geometric genus of their possible singularities. Moreover, we will construct new examples of this type of surfaces.
\end{abstract}

\maketitle

\section[Introduction]{Introduction} Let $g\geq3$. It is well-known that a \emph{Prym curve} is a pair $(C,\alpha)$, where $C$ is a smooth genus $g=p_g(C)$ curve and $\alpha$ is a non-zero $2-$torsion point of $\operatorname{Pic}^0(C)$. In the following, we will consider the so called \emph{Prym-canonical map}, that is the rational map $$\phi_{|\omega_C(\alpha)|}:C\dashrightarrow \mathbb{P}^{g-2}$$ defined by $|\omega_C(\alpha)|$. In general, the pair $(C, \omega_C(\alpha))$ is called \emph{Prym-canonical curve}. The complete linear system $|\omega_C(\alpha)|$ is base point free unless $C$ is hyperelliptic and $\alpha\simeq \mathcal{O}_{C}(p-q)$, with $p$ and $q$ ramification points of the $g^1_2$. Moreover, it defines an embedding if and only if $C$ does not have a $g^1_4$ such that $\alpha\sim \mathcal{O}_{C}(a+b-x-y)$, where $2(a+b)$ and $2(x+y)$ are members of the $g^1_4$ (see \cite{CDGK}, Lemma $2.1$). If $\phi_{|\omega_C(\alpha)|}$ is an embedding, we say that $C \simeq \phi(C) \subset P^{g-2}$ is a \emph{Prym-canonical (embedded) curve}. If $g<5$, then the Prym-canonical map cannot be an embedding, as observed in \cite{CDGK}, so we will work only with $g\geq5$.

\par\bigskip\noindent We say that a \emph{surface $X$ has Prym-canonical hyperplane sections} if it can be birationally realized in some projective space $\mathbb{P}^{g-1}$, for $g\geq5$, such that a general hyperplane section $C$ of $X$ is a smooth Prym-canonical (embedded) curve of genus $g$.

\par\bigskip\noindent In this paper we will analyze the complex projective surfaces with Prym-canonical hyperplane sections up to birational equivalence. In particular, Section $2$ is devoted to studying the first properties regarding surfaces with Prym-canonical hyperplane sections. We will show that these surfaces can be birationally equivalent to ruled surfaces or to $\mathbb{P}^2$ or to Enriques surfaces. In any case, there is only one effective antibicanonical divisor $W'$ on $X'$, the minimal resolution of the singularities of $X$, and, if $\pi:X'\rightarrow X$, then the antibicanonical divisor of $X'$ is contracted by $\pi$ and every singularity $x\in X$ such that $\pi^{-1}(x)$ does not meet $\operatorname{supp}(W')$ is a rational double point. We will also show that surfaces with Prym-canonical hyperplane sections birationally equivalent to non-rational ruled surfaces over a base curve of genus $q>0$ have non rational singularities on them such that the sum of the geometric genus of their singularities is $q$.

\par\bigskip\noindent At the best of our knowledge, the only known examples of surfaces with Prym-canonical hyperplane sections are the Enriques surfaces and a surface in $\mathbb{P}^5$ of degree $10$ obtained as image of the blowing up of $\mathbb{P}^2$ in the $10$ nodes of an irreducible rational plane curve of degree $6$. In Section $3$ we will construct new examples of these surfaces. In particular, we will construct four new examples, one birationally equivalent to an elliptic ruled surface, another birationally equivalent to a ruled surface over a base curve of genus $q\geq3$, again another one birationally equivalent to a rational ruled surface and finally, we will construct a new example of surface with Prym-canonical hyperplane sections birationally equivalent to $\mathbb{P}^2$.

\par\bigskip \textit{Acknowledgements.} The results of this paper are contained in my PhD-thesis. I would like to express my deepest gratitude to my three advisors, Ciro Ciliberto, Concettina Galati and Andreas Leopold Knutsen, for their useful and indispensable advice and I would also like to acknowledge PhD-funding from the Department of Mathematics and Computer Science of the University of Calabria and funding from Research project "Families of curves: their moduli and their related varieties" (CUP E81|18000100005, P.I. Flaminio Flamini) in the framework of Mission Sustainability 2017 - Tor Vergata University of Rome.

\par\bigskip
 \section{Preliminary results}
We recall that a surface $X\subseteq \mathbb{P}^{g-1}$ is a surface with Prym canonical hyperplane sections if its general hyperplane section $C$ is a Prym canonical embedded curve. We start with the following remarks.

\begin{rems}
\end{rems}
\begin{itemize}
  \item The generic hyperplane section of $X$ is irreducible and smooth, whence $X$ has at most isolated singularities.
  \item Let $\pi:X'\rightarrow X$ be the minimal resolution of singularities of $X$ and let $C'=\pi^*C$ be the inverse image of a general hyperplane section. For a general $C' \in |\pi^*C|$, we have that $C' \cong C$ and $\mathcal{O}_{C'}(C') \cong \mathcal{O}_C(1) \cong \omega_{C'}(\alpha)$, with $\alpha$ a non trivial two-torsion element of $\operatorname{Pic}^0(C')$. By the adjunction formula and because $X'$ is smooth, we can say that $\alpha=-K_{X'}|_{C'}$, in particular $K_{X'}\cdot C'=0$.
 \end{itemize}

 \par\bigskip\noindent From now on, we will assume that $C$ is projectively normal with respect to its embedding in $\mathbb{P}^{g-2}$. For example, if $\phi_{|\omega_C(\alpha)|}:C\hookrightarrow \mathbb{P}^{g-2}$ is a Prym-canonical embedding and the Clifford index $\operatorname{Cliff}(C)$$\geq3$, then $C$ is projectively normal with respect to the given embedding (see Theorem $1$, \cite{GL}). In general, there are also projectively normal curves $C\subseteq\mathbb{P}^{g-2}$ with $\operatorname{Cliff}(C)<3$.

 \par\bigskip\noindent We state some general properties regarding surfaces with Prym-canonical hyperplane sections.
 \begin{thm} \label{first properties}
   Let $X$ be a surface with Prym-canonical hyperplane section $C$ of genus $g\geq 5$ and let $\pi:X'\rightarrow X$ be the minimal resolution of its singularities. If $C$ is projectively normal with respect to its embedding in $\mathbb{P}^{g-2}$, then:
  \begin{itemize}
    \item $h^1(\mathcal{O}_{X}(n))=0$ and $h^2(\mathcal{O}_{X}(n))=0$ for any $n\geq0$, in particular $h^1(\mathcal{O}_X)=0$ and $h^2(\mathcal{O}_X)=0$, whence $p_a(X)=0$;
    \item $X$ is projectively normal;
    \item the Kodaira dimension $\kappa(X')$ equals to $-\infty$ or $0$;
    \item $\deg(X)=2g-2$.
  \end{itemize}
\end{thm}
\begin{proof}
  \begin{itemize}
    \item By assumption, $C$ is a projectively normal curve in $\mathbb{P}^{g-2}$, thus $$H^0(\mathcal{O}_{\mathbb{P}^{g-1}}(n))\twoheadrightarrow H^0(\mathcal{O}_C(n)),$$ for any $n\geq0$. As a consequence the map $H^0(\mathcal{O}_{X}(n))\rightarrow H^0(\mathcal{O}_C(n))$ is surjective for any $n\geq0$.

        \noindent Let us consider the exact sequence \begin{equation}\label{exact2}
                                                       0\rightarrow \mathcal{O}_X(n-1)\rightarrow \mathcal{O}_X(n)\rightarrow \mathcal{O}_C(n)\rightarrow0.
                                                     \end{equation}

                 \noindent Then, the second part of the long exact sequence associated with (\ref{exact2}) is
                 $$0\rightarrow H^1(\mathcal{O}_X(n-1))\rightarrow H^1(\mathcal{O}_X(n))\rightarrow H^1(\mathcal{O}_C(n))\rightarrow  $$
                 \begin{equation}\label{exact22}
                  \rightarrow H^2(\mathcal{O}_X(n-1))\rightarrow H^2(\mathcal{O}_X(n))\rightarrow 0.
                 \end{equation}

         \noindent By Serre's Theorem, there is a sufficiently large $n_0$ such that $h^1(\mathcal{O}_X(n))=0$, for any $n\geq n_0$. From the exact sequence (\ref{exact22}) and applying descending induction on $n$, we obtain that $h^1(\mathcal{O}_X(n))=0$, for any $n\geq0$.

    \par\noindent It is clear that $H^1(\mathcal{O}_C(n))=H^1(\mathcal{O}_C(n(K_C+\alpha)))$. For $n=1$, we have that $h^1(\mathcal{O}_C(K_C+\alpha))=h^0(\mathcal{O}_C(-\alpha))=0$. Moreover, because $\deg(K_C+\alpha)=2g-2$, then $\deg(n(K_C+\alpha))>2g-2$ for $n\geq2$, so $h^1(\mathcal{O}_C(n(K_C+\alpha)))=0$ (see \cite{H}, Example $IV.1.3.4$). Again by Serre's Theorem and applying descending induction on $n$, from the long exact sequence (\ref{exact22}) we can conclude that $h^2(\mathcal{O}_X(n))=0$, for any $n\geq0$.

    \item To prove that $X$ is projectively normal, it is enough to show that $X$ is normal and that the map $H^0(\mathcal{O}_{\mathbb{P}^{g-1}}(n))\rightarrow H^0(\mathcal{O}_X(n))$ is surjective for any $n\geq0$.

       \noindent Let $\eta:\widetilde{X}\rightarrow X$ be the normalization of $X$. We consider the following exact sequence on $X$: \begin{equation}\label{exact1}
                          0\rightarrow \mathcal{O}_X\rightarrow \eta_{*}\mathcal{O}_{\widetilde{X}}\rightarrow F\rightarrow 0,
                        \end{equation} where $\operatorname{supp}(F)\subset \operatorname{Sing}(X)$. Since $X$ has isolated singularities, then \linebreak
                         $F\cong H^0(F)=\oplus_{i=1}^{s}(\widetilde{\mathcal{O}_i}/\mathcal{O}_i)$, where $\mathcal{O}_i$ is the local ring of $x_i$ on $X$ and $(\eta_{*}\mathcal{O}_{\widetilde{X}})_{x_i}=\widetilde{\mathcal{O}_i}$ is the normalization of $\mathcal{O}_i$ in the function field of $X$, with $\operatorname{Sing}(X)=\{x_1,...,x_s\}$.

        \noindent We know that $H^0(\mathcal{O}_X)=k$ because $X$ is irreducible. By the properties of pushforward and because $\widetilde{X}$ is still irreducible, it is obvious that $H^0(\eta_{*}\mathcal{O}_{\widetilde{X}})\cong H^0(\mathcal{O}_{\widetilde{X}})\cong k$. Moreover $h^1(\mathcal{O}_X)=0$ by the previous part of this Proposition. For the long exact sequence associated with (\ref{exact1}), we have that $h^0(F)=0$. By definition of $F$, it is true that $\widetilde{\mathcal{O}_i}\cong \mathcal{O}_i$, for any $i=1,...,s$. We conclude that $X$ is normal.

       \noindent The surjectivity of $H^0(\mathcal{O}_{\mathbb{P}^{g-1}}(n))\rightarrow H^0(\mathcal{O}_X(n))$ is trivial for $n=0$. Let us consider the following diagram, where $H$ is a general hyperplane in $\mathbb{P}^{g-1}$ and $C=X\cap H$:

    \bigskip \begin{flushleft}
             \begin{tikzcd}
        0 \arrow[r]
					& H^0(\mathcal{O}_{\mathbb{P}^{g-1}}(n-1)) \arrow[r] \arrow[d, "r_1"]
						& H^0(\mathcal{O}_{\mathbb{P}^{g-1}}(n)) \arrow[r]\arrow[d,"r_2"]
							& H^0(\mathcal{O}_H(n)) \arrow[r]\arrow[d,"r_3"]
							& 0 \\
				0 \arrow[r]
					& H^0(\mathcal{O}_X(n-1)) \arrow[r]
						& H^0(\mathcal{O}_X(n)) \arrow[r]
			& H^0(\mathcal{O}_C(n)) \arrow[r]
					& 0 \\
       \end{tikzcd}
             \end{flushleft}

   \par\noindent We observe that $r_3$ is surjective because $C$ is projectively normal and $r_1$ is surjective by the inductive hypothesis. Then $r_2$ is also surjective and the claim is proved.
    \item Consider the following exact sequence:
    $$0\rightarrow \mathcal{O}_{X'}(-C'+mK_{X'})\rightarrow \mathcal{O}_{X'}(mK_{X'})\rightarrow \mathcal{O}_{C'}(m K_{X'}))\rightarrow0.$$
     Since $-K_{X'}|_{C'}\sim \alpha$, for $\alpha$ a non-zero two torsion element of $C'$, then we have that $(-C'+mK_{X'})\cdot C'=-C'^2=2-2g<0$. Whence \linebreak$h^0(\mathcal{O}_{X'}(-C'+mK_{X'}))=0$ otherwise, if this divisor was effective, it would be a fixed component of $|C'|$ that is a linear system without base locus by definition. At the same time

     $$ h^0(\mathcal{O}_{C'}(m(K_{X'})))=
\bigg\{
\begin{array}{lll}
 0& if \hspace{0.3cm}m\hspace{0.3cm} odd \\
1 & if \hspace{0.3cm}m\hspace{0.3cm} even. \\
\end{array}
$$

     \bigskip\noindent Consequently the plurigenus $P_m(X'):=h^0(\mathcal{O}_{X'}(mK_{X'}))\leq h^0(\mathcal{O}_{C'}(m K_{X'}))\linebreak\leq1$. Then the Kodaira dimension $\kappa(X')=-\infty$ or $0$.

    \item We have $\deg(X)=\deg(C)=C^2=\deg(K_C+\alpha)=2g-2$.
  \end{itemize}
\end{proof}

\begin{rems}\label{minima Enri}
\end{rems}
\begin{enumerate}
  \item Since the Kodaira dimension is a birational invariant for smooth varieties, if $\kappa(X')=-\infty$, then the minimal model $X''$ of $X'$ is a ruled surface or $\mathbb{P}^2$ (see \cite{H}, Theorem $V.6.1$).
  \item If $\kappa(X')=0$, then $X'$ is a minimal Enriques surface. Let us show this.

  \noindent Let us suppose that $X'$ is not minimal. Then there is a $(-1)-$curve $E'$ on $X'$. Because $\kappa(X')=0$, let $m>0$ be such that $|mK_{X'}|$ contains only one effective divisor $D'$. It is obvious that $E'$ is a component of $D'$. Now $\mathcal{O}_{C'}(D')\cong \mathcal{O}_{C'}(mK_{X'})\cong \mathcal{O}_{C'}$ because $-K_{X'}|_{C'}$ is a non-zero two torsion element and $m$ is even as seen in the previous Proposition. Therefore $D'$ and consequently $E'$ are contracted to a point on $X$ by $\pi$, contradicting the minimality of the resolution $\pi$.

\noindent
    It is true that $12K_{X'}\sim0$ because $X'$ is minimal and $\kappa(X')=0$ (see \cite{H}, Theorem $V.6.3$). By the classification of minimal surfaces, there is a smallest $m\geq1$ such that $mK_{X'}\sim0$ and the possibilities are $m=\{1,2,3,4,6\}$ (see \cite{E} and \cite{CE}). If $m=1$, then $K_{X'}\sim0$, whence $\mathcal{O}_{C'}(C')\cong \mathcal{O}_{C'}(K_{C'}-K_{X'})\cong \mathcal{O}_{C'}(K_{C'})$. This is not possible because $C'$ is a Prym-canonical curve. So we exclude the cases in which $X'$ is a $K3$ surface or an abelian surface.

    \noindent If $X'$ was a hyperelliptic surface, it would not contain curves with negative self-intersection, then we would have $X=X'$ smooth by Mumford's Theorem (see \cite{MUM}, Chapter $1$). By definition, a hyperelliptic surface is irregular, contradicting the first point of Proposition \ref{first properties}. In conclusion $X'$ is a minimal Enriques surface.

\end{enumerate}

\par\bigskip We want to determine the possible singularities on $X$ surface with Prym-canonical hyperplane sections. The following Proposition determines the geometric genus of the singularities that occur on $X$.

\begin{prop}\label{sum sing} With the same assumptions as before, if $\operatorname{Sing}(X)=\{x_1,...,x_s\}$ is the locus of the singular points of $X$, then:\begin{itemize}
    \item if $X$ is birationally equivalent to an Enriques surface or $\mathbb{P}^2$, then $X$ can only contain rational points as singularities;
    \item if $X$ is birationally equivalent to a ruled surface $X''$ over a base curve of genus $q\geq0$, then $\sum_{i=1}^{s}p_g(x_i)=q$, where $p_g(x_i)$ is the geometric genus of the singular point $x_i$.
  \end{itemize}
\end{prop}
\begin{proof} These results can be obtained using the following exact sequence, which one gets from the Leray spectral sequence for the sheaf $\mathcal{O}_{X'}$ and the morphism $\pi$ (see \cite{GH}, pag. 462): $$0\rightarrow H^1(\mathcal{O}_X)\rightarrow H^1(\mathcal{O}_{X'})\rightarrow H^0(R^1\pi_*\mathcal{O}_{X'})\rightarrow H^2(\mathcal{O}_X)\rightarrow ...$$ We know that $h^1(\mathcal{O}_{X})=h^2(\mathcal{O}_X)=0$ by Theorem \ref{first properties}, so $$\sum_{i=1}^{s}p_g(x_i)=h^0(R^1\pi_*\mathcal{O}_{X'})=h^1(\mathcal{O}_{X'}),$$ where $p_g(x_i)$ is the geometric genus of the singular point $x_i$.

\begin{itemize}
  \item If $X$ is birationally equivalent to an Enriques surface, then $X'$ is a minimal Enriques surface by Remark \ref{minima Enri}. So $h^1(\mathcal{O}_{X'})=0$ by definition and $\sum_{i=1}^{s}p_g(x_i)=0$, whence $X$ can only contain rational singularities.

  \noindent If $X$ is birationally equivalent to $\mathbb{P}^2$, then it is clear that $h^1(\mathcal{O}_{X'})=0$ and $\sum_{i=1}^{s}p_g(x_i)=0$. This proves the first part of this Proposition.
  \item If $X$ is birationally equivalent to a ruled surface $X''$ over a base curve of genus $q\geq0$, then $h^1(\mathcal{O}_{X''})=q(X'')=q$. Consequently $q(X'')=q(X')=q$. So $\sum_{i=1}^{s}p_g(x_i)=h^1(\mathcal{O}_{X'})=q$.
\end{itemize}\end{proof}

\par\bigskip\noindent The following results give other information regarding the singularities that occur on $X$. First of all, we sketch the proof of a preliminary lemma.

\begin{lemma}
  Let $X$ be a surface with Prym-canonical hyperplane section $C$ and let $\pi:X'\rightarrow X$ be the minimal resolution of singularities of $X$. Then $\pi_*(2K_{X'})\sim 0$.
\end{lemma}
\begin{proof}
  Let $C_m\in |mC|$ be smooth satisfing $C_m\cap\operatorname{Sing}(X)=\varnothing$. We put \linebreak$C'_m=\pi^{*}(C_m)$. We have that $-2K_{X'}\cdot C'_m\sim-2mK_{X'}\cdot C'\sim0$ in the Chow Ring $A(X')$. Since rational equivalence and linear equivalence coincide on a curve, then $-2K_{X'}|_{C'_m}\sim 0$, for any $m\geq1$.

  \par\bigskip\noindent
  We also observe that $\pi_*(2K_{X'})|_{C_m}\sim0$. Indeed, since $\pi$ is an isomorphism in a neighbourhood of $C'_m$, then $\pi_*\mathcal{O}_{C'_m}\cong \mathcal{O}_{C_m}$ and $\mathcal{O}_{C'_m}\cong \pi^*(\mathcal{O}_{C_m})$. Using the projection formula (see \cite{H}, Exe $II.5.1$) and the previous results, we obtain that \linebreak $\mathcal{O}_{C_m}\cong \pi_*(\mathcal{O}_{C'_m})\cong \pi_*(2K_{X'}\otimes \mathcal{O}_{C'_m})=\pi_*(2K_{X'}\otimes \pi^*\mathcal{O}_{C_m})\cong \pi_*(2K_{X'})\otimes \mathcal{O}_{C_m}$.

  \bigskip\noindent By a known result of Zariski (\cite{Z}, Theorem $4$), if $\pi_*(2K_{X'})|_{C_m}\sim0$ for any $m$, then there exists a divisor $D$ such that $D\sim \pi_*(2K_{X'})$ and $D|_{C_m}\sim 0$. Since $C_m$ is very ample on $X$, then $D\sim0$. So $\pi_*(2K_{X'})\sim0$. This proves the lemma.

\end{proof}

\par\bigskip
\begin{thm}\label{unique abc}
  The dimension $\dim|-2K_{X'}|=0$, in particular, if $W'$ is the effective antibicanonical divisor on $X'$, then either $W'\sim 0$ or $\operatorname{supp}(W')=\pi^{-1}({x_1,...,x_r})$ for certain singularities $x_i\in X$, for $i=1,..,r$.
     \end{thm}
\begin{proof}
   Since $\pi_*(2K_{X'})\sim 0$ by the previous Lemma, then either $2K_{X'}\sim 0$ or there is a bicanonical divisor $2K_{X'}$ on $X'$ with support in $\pi^{-1}(\operatorname{Sing} (X))$. In the latter case, let $2K_{X'}=\sum m_iF_i -\sum n_jG_j$ be the decomposition in reduced and irreducible components, with $m_i,n_j\in \mathbb{N}_{>0}$ and $F_i\neq G_j$, for all $i,j$. Let $F=\sum m_i F_i$ and $G=\sum n_jG_j$.

        \noindent Suppose $F\neq0$. By Mumford's Theorem (see \cite{MUM}, Chapter $1$), we have that $F_i^2<0$ for any $i$ and, because the intersection form on $\pi^{-1}(\operatorname{Sing}(X))$ is negative definite, also $F^2<0$. So there is an $i_0$ such that $F\cdot F_{i_0}<0$. Up to renaming the index, we suppose that $i_0=1$. It is obvious that $F_1\cdot G\geq0$. Since $F_1$ is an irreducible component, then $$0\leq p_a(F_1)=1+\frac{1}{2}F_1\cdot (F_1+K_{X'})=1+\frac{1}{2}F_1^2+\frac{1}{4}F_1\cdot F-\frac{1}{4}F_1\cdot G.$$

        \noindent The only possibility is $F_1^2=-1$. Thus $F_1$ is a $(-1)-$curve, contradicting the minimality of $\pi$. Hence $F=0$.

        \noindent Thus either $2K_{X'}\sim 0$ or there are effective antibicanonical divisors with support in $\pi^{-1}(\operatorname{Sing}(X))$. Then $\dim|-2K_{X'}|=0$.

        \bigskip\noindent  Let $|-2K_{X'}|=\{W'\}$. To conclude the proof, we have only to show that, if $x\in \operatorname{Sing}(X)$ is such that $\pi^{-1}(x)$ meets $\operatorname{supp}(W')$, then $\pi^{-1}(x)$ does not contain curves which are not part of $\operatorname{supp}(W')$.

    \noindent Suppose that there is an irreducible curve $E\subset \pi^{-1}(x)$ which is not part of $\operatorname{supp}(W')$. Since $X$ is normal by Theorem \ref{first properties}, Point $2.$, then $\pi^{-1}(x)$ is connected, so we can assume that $E$ intersects $W'$ and $E\cdot W'>0$. Then $$0\leq p_a(E)=1+\frac{1}{2}E^2+\frac{1}{2}E\cdot K_{X'}=1+\frac{1}{2}E^2-\frac{1}{4}E\cdot W'.$$ Again by Mumford's Theorem, we have $E^2<0$. So the only possible case for which the previous inequality is valid is: $E\cdot W'=2$, $E^2=-1$ and $p_a(E)=0$. This contradicts the minimality of $\pi$.
\end{proof}

\begin{rem}
  By the previous Theorem, we observe that, if $X$ is smooth, then $X=X'$ and $W'\sim 0$. Since $p_a(X)=p_g(X)=0 $ by Theorem \ref{first properties}, Point $1.$, then $X$ is an Enriques surface by \cite{H}, Theorem $V.6.3$.
\end{rem}

\par\bigskip
 \begin{lemma}\label{supp antibican}
   If $W'$ is the unique effective antibicanonical divisor on $X'$, then:

    \noindent a singularity $x\in X$ such that $\pi^{-1}(x)$ does not meet $\operatorname{supp}(W')$ is a rational double point.
     \end{lemma}

    \begin{proof} 

        Let $x\in X$ be a singularity such that $\pi^{-1}(x)$ does not meet $\operatorname{supp}(W')$. Let $T$ be an irreducible component of the connected component $\pi^{-1}(x)$, then $T\cdot W'=0$. So $$0\leq p_a(T)=1+\frac{1}{2}T^2+\frac{1}{2}T\cdot K_{X'}=1+\frac{1}{2}T^2-\frac{1}{4}T\cdot W'=1+\frac{1}{2}T^2.$$ By Mumford's Theorem $T^2<0$, so the only possible case for which the inequality above is valid is: $T^2=-2$ and $p_a(T)=0$. Then all the irreducible components of $\pi^{-1}(x)$ are smooth rational curves with self-intersection $-2$. We can call these curves $E_i$, for $i=1,...,n$.

        \par\bigskip\noindent We can prove that $x$ must be a rational singularity using \cite{M}, Proposition-Definition $2.1$. Let $Z_0=\sum_{i=1}^{n}a_i\cdot E_i$, for $a_i\geq0$ not all zero, the fundamental cycle associated with $x$.
        First of all, $p_a(Z_0)=1+\frac{1}{2}Z_0^2+\frac{1}{2}Z_0\cdot K_{X'}$. Now $Z_0\cdot K_{X'}=-\frac{1}{2}Z_0\cdot W'=-\frac{1}{2}\sum_{i=1}^{n}a_i E_i\cdot W'=0$, while $Z_0^2<0$ since $Z_0$ is contracted by $\pi$. So $p_a(Z_0)<1$.

        \noindent By \cite{Tom}, Lemma $1.1$, we have that $p_a(E_i)\leq p_a(Z_0)$ for every $E_i$ contained in $Z_0$. Since $p_a(E_i)=0$ as computed before, then the only possible case is $p_a(Z_0)=0$, so $x$ is a rational singularity.

        \par\bigskip\noindent In conclusion, $x\in X$ is a rational double point. Indeed, by the adjunction formula, the self-intersection $Z_0^2=2p_a(Z_0)-2-Z_0\cdot K_{X'}=2p_a(Z_0)-2=-2$.
  \end{proof}

  \begin{rem}
    We have already seen that, if $X$ is birationally equivalent to an Enriques surface, then $X$ can only contain rational singularities.

    \noindent Moreover, by the previous Lemma, we conclude that it can only contain rational double points as singularities.
  \end{rem}

\section{Examples}
\noindent In this chapter, we will construct examples of surfaces with Prym-canonical hyperplane sections.
\subsection{Surfaces with Prym-canonical hyperplane sections birationally equivalent to ruled surfaces}
\par\bigskip\noindent Let us fix some notation about minimal smooth ruled surfaces in which we will follow \cite{H}, Chapter $V.2$.

\par\bigskip\noindent If $X''$ is a minimal smooth ruled surface and $p:X''\rightarrow \Gamma$ is the natural map on the base curve $\Gamma$ of genus $q\geq 0$, then $X''=\mathbb{P}_{\Gamma}(\mathcal{E})$, where $\mathcal{E}$ is a normalized locally free sheaf of rank $2$ on $\Gamma$.

\noindent Let $\wedge^2(\mathcal{E})=\mathcal{O}_{\Gamma}(D)$, for $D\in \operatorname{Div}(\Gamma)$. The integer $e=-\deg(D)$ is an invariant of $X''$.

\par\noindent Let $C_0$ be a section of $p$ such that $\mathcal{O}_{X''}(C_0)=\mathcal{O}_{X''}(1)$. Then $C_0^2=-e$. Moreover, we recall that the canonical divisor $K_{X''}\sim -2C_0+(K_{\Gamma}+D)\cdot f$, where $(K_{\Gamma}+D)\cdot f$ denotes the divisor $p^*(K_{\Gamma}+D)$ by abuse of notation, with $K_{\Gamma}+D$ a divisor on $\Gamma$.

\par\noindent Finally, if $\mathcal{E}$ is decomposable, i.e. $\mathcal{E}=\mathcal{O}_{\Gamma}\oplus\mathcal{O}_{\Gamma}(D)$, then we will denote by $C_1$ a fixed section of $p$ disjoint from $C_0$. Thus $C_1^2=e$ and $C_1\sim C_0-D\cdot f$.

\subsubsection{\textbf{The minimal model is a non-rational ruled surface}}
\par\bigskip\noindent We start recalling a simple example of surface with Prym-canonical hyperplane sections.

\par\bigskip\noindent Let $X''=\mathbb{P}_{\Gamma}(\mathcal{O}_{\Gamma}\oplus \mathcal{O}_{\Gamma}(D))$ be a minimal ruled surface over a base curve $\Gamma$ of genus $q\geq5$, with $D\sim -K_{\Gamma}+\alpha$, where $\alpha$ is a not trivial two torsion divisor and $K_{\Gamma}$ is the canonical divisor of $\Gamma$, and let $L''=|C_1|$ be a linear system on $X''$. By \cite{CDGK}, Lemma $2.1$, if $\Gamma$ is non-hyperelliptic and it does not admit $g^1_4$, then a general hyperplane section of $X''$ is Prym-canonically embedded. The images of fibres of $X''$ by $i_{L''}$ are lines since $C''\cdot f =C_1\cdot f=1$. It is not difficult to prove that $X$ has only one singularity, so the map $i_{L''}:X''\dashrightarrow \mathbb{P}^{q-1}$ defined by the linear system $L''$ is such that $X=i_{L''}(X'')$ is a cone on a Prym-canonical embedded curve and if a general hyperplane section $C$ of $X$ is projectively normal, then the geometric genus of the only singularity is $q$ (see Proposition \ref{sum sing}).

\begin{rem}\label{rem cone} Let us consider $L''\subseteq|aC_0+\Delta\cdot f|$, for $a\geq2$ and $\Delta\in \operatorname{Pic}(\Gamma)$. If $X''=\mathbb{P}_{\Gamma}(\mathcal{O}_{\Gamma}\oplus\mathcal{O}_{\Gamma}(D))$ is a minimal ruled surface over a base curve $\Gamma$ of genus $q\geq5$, then the images of fibres of $X''$ by the map $i_{L''}$ associated with $L''$ are not lines since $C''\cdot f=(aC_0+\Delta\cdot f)\cdot f=a>1$, for $C''$ a general curve in $L''$. Furthermore there is not another family of rational curves on $X''$ mapped into lines by $i_{L''}$ because the genus of the base curve $\Gamma$ is $q>0$. Indeed, by the Riemann-Hurwitz formula (see \cite{H}, Corollary $IV.2.4$), there is not a curve of genus $0$ (not equal to a fibre) on a ruled surface over a base curve of genus $q>0$. Hence we never obtain $X=i_{L''}(X'')$ as a cone.
\end{rem}

\par\bigskip Now we focus our attention on surfaces with Prym-canonical hyperplane sections birationally equivalent to ruled surfaces $X''$ over a non-hyperelliptic base curve of genus $q\geq3$.

\begin{prova}\label{example}
  Let $X''=\mathbb{P}_{\Gamma}(\mathcal{O}_{\Gamma}\oplus \mathcal{O}_{\Gamma}(D))$ be a minimal ruled surface over a non-hyperelliptic smooth base curve $\Gamma$ of genus $q\geq3$, for $D\in \operatorname{Div}(\Gamma)$. Let $L''=|aC_1|$ be a linear system with $a\geq2$. If $D\sim -K_{\Gamma}+\alpha$, for $\alpha$ a non-zero two-torsion divisor, then the image $X=i_{L''}(X'')\subseteq\mathbb{P}^{a^2(q-1)}$ of the morphism associated with $L''$ has Prym-canonical hyperplane sections and only one singularity. In particular, if a general hyperplane section $C$ of $X$ is projectively normal, then the geometric genus of the only singularity $x$ is $p_g(x)=q$.\end{prova}

\begin{proof}
  It is easy to prove that the linear system $L''=|aC_1|$ is base-point free using \cite{FP}, Proposition $36$, for any $a\in \mathbb{N}_{\geq2}$. Moreover, since $(C'')^2>0$, for $C''$ a general element of $L''$, then, by Bertini's Theorem, $C''$ is smooth and irreducible.

  \par\bigskip\noindent $\mathrm{CLAIM \hspace{0.2cm}1:}$ We have that $-K_{X''}$ is not effective, while $-2K_{X''}$ is.

  \begin{proof} We know that $h^0(\mathcal{O}_{X''}(-2K_{X''}))=h^0(\mathcal{O}_{X''}(4C_0))=h^0(\mathcal{O}_{\Gamma})+h^0(\mathcal{O}_{\Gamma}(D))+h^0(\mathcal{O}_{\Gamma}(2D))+h^0(\mathcal{O}_{\Gamma}(3D))+h^0(\mathcal{O}_{\Gamma}(4D))$ by \cite{FP}, Lemma $35$. Because $\deg(D)=2-2q<0$, then $h^0(\mathcal{O}_{X''}(-2K_{X''}))=1$ and $-2K_{X''}$ is effective.

  \par\bigskip\noindent On the other hand, we have that $h^0(\mathcal{O}_{X''}(-K_{X''}))=h^0(\mathcal{O}_{X''}(2C_0-(K_{\Gamma}+D)\cdot f))= h^0(\mathcal{O}_{\Gamma}(-K_{\Gamma}-D))+h^0(\mathcal{O}_{\Gamma}(-K_{\Gamma}))+h^0(\mathcal{O}_{\Gamma}(-K_{\Gamma}+D))$ by \cite{FP}, Lemma $35$. Since $\deg(D-K_{\Gamma})=4-4q<0$ and $\deg(-K_{\Gamma})<0$, then $h^0(-K_{X''})=\linebreak h^0(\mathcal{O}_{\Gamma}(-K_{\Gamma}-D))=h^0(\mathcal{O}_{\Gamma}(-\alpha))=0$, so $-K_{X''}$ is not effective. \end{proof}

  \par\bigskip\noindent $\mathrm{CLAIM \hspace{0.2cm}2:}$ We prove that $\mathcal{O}_{C''}(-K_{X''})\ncong \mathcal{O}_{C''}$ while $\mathcal{O}_{C''}(-2K_{X''})\cong \mathcal{O}_{C''}$, where $C''\in L''$ is a general curve.

\begin{proof} Since $-2K_{X''}$ is effective as seen in Claim $1$ and $C''\cdot(-2K_{X''})=\linebreak(aC_0-aD\cdot f)\cdot (4C_0)=4aC_0^2-4a\deg(D)=0$, then $\mathcal{O}_{C''}(-2K_{X''})\cong \mathcal{O}_{C''}$.

\par\bigskip\noindent Clearly also $C''\cdot(-K_{X''})=0$. Since $h^0(\mathcal{O}_{X''}(-K_{X''}))=0$ as seen in Claim $1$, if we prove that $h^1(\mathcal{O}_{X''}(-K_{X''}-C''))=0$, then, from the exact sequence $$0\rightarrow \mathcal{O}_{X''}(-K_{X''}-C'')\rightarrow \mathcal{O}_{X''}(-K_{X''})\rightarrow \mathcal{O}_{C''}(-K_{X''})\rightarrow 0,$$ we have that $h^0(\mathcal{O}_{C''}(-K_{X''}))=0$, which implies that $\mathcal{O}_{C''}(-K_{X''})\ncong \mathcal{O}_{C''}$.

\noindent Thus, by Serre Duality, we have that $h^1(\mathcal{O}_{X''}(-K_{X''}-C''))=h^1(\mathcal{O}_{X''}(2K_{X''}+C''))$. If we prove that $K_{X''}+C''$ is ample, then, by the Kodaira vanishing Theorem (see \cite{H}, Remark $III.7.15$), we have that $h^1(\mathcal{O}_{X''}(2K_{X''}+C''))=0$ and the claim is proved.

\noindent By \cite{H}, Proposition $V.2.20$, if $a>2$, then $K_{X''}+C''$ is ample and the claim is satisfied, instead, if $a=2$, we have that $K_{X''}+C''$ is not ample. About that, let us suppose that $\mathcal{O}_{C''}(-K_{X''})\cong \mathcal{O}_{C''}$. Then the image of $X''$ by $i_{L''}$ is a surface with canonical hyperplane sections. By \cite{Epe1}, Corollary $5.4$, $X''$ contains only one effective anticanonical divisor. This contradicts Claim $1$, then this claim is also satisfied for $a=2.$\end{proof}

  \par\bigskip\noindent
\noindent We know that $h^0(\mathcal{O}_{\Gamma})=1$. Using the Riemann-Roch Theorem, we also have that $h^0(\mathcal{O}_{\Gamma}(-D))=h^0(\mathcal{O}_{\Gamma}(\alpha))+2q-2+1-q=(2q-2)+1-q$ and $h^0(\mathcal{O}_{\Gamma}(-mD))=m(2q-2)+1-q$, for any $m\in \mathbb{N}_{>1}$. Hence, using \cite{FP}, Lemma $35$, we obtain that $$h^0(\mathcal{O}_{X''}(L''))=(1+...+a)(2q-2)+a(1-q)+1=$$$$=a(a+1)(q-1)+a(1-q)+1=a^2(q-1)+1.$$

%
%
%

  \noindent So $X=i_{L''}(X'')$ is contained in $\mathbb{P}^{a^2(q-1)}$, for $a^2(q-1)\geq4\cdot 2=8$.

\par\bigskip
  \par\bigskip\noindent $\mathrm{CLAIM \hspace{0.2cm}3:}$ It remains to show that $L''$ defines a birational morphism $i_{L''}$, in particular an isomorphism outside $C_0$. So $i_{L''}|_{C''}$ is a Prym-canonical embedding, for $C''\in L''$ a general divisor.

\begin{proof}
  \par\bigskip We can prove that $L''$ defines a birational map, in particular an isomorphism outside $C_0$, using \cite{FP}, Theorem $38$. Indeed, this happens if $-aD$ is very ample and if $|-aD+D|$, $|-aD+(a-1)D|=|-D|$ and $|-aD+aD|$ are base-point free.

  \noindent The last case is trivial. By \cite{H}, Corollary $IV.3.2$, since $\deg(-aD)=a(2q-2)\geq 4q-4=2q+(2q-4)\geq2q+1$, then $-aD$ is very ample. Again by \cite{H}, Corollary $IV.3.2$, if $a\geq3$, since $\deg(-aD+D)=a(2q-2)+(2-2q)=(a-1)(2q-2)\geq 4q-4=2q+(2q-4)\geq2q$, then $|-aD+D|$ is base-point free.

  \noindent It remains to show that $|-D|$ is base-point free. Since $D\sim -K_{\Gamma}+\alpha$, if $|-D|$ has base points, then $\Gamma$ must be hyperelliptic (see \cite{CDGK}, Lemma $2.1$). This contradicts our assumptions. So $L''$ defines an isomorphism outside $C_0$, called $i_{L''}$ .

  \bigskip\noindent A general $C''\sim aC_1$ in $L''$ is disjoint from $C_0$ by definition, then $$i_{L''}|_{C''}:C''\rightarrow \mathbb{P}^{a^2(q-1)-1}$$ is an embedding. By the adjunction formula, we have that $$L''|_{C''}\cong K_{C''}-K_{X''}|_{C''}$$ but in Claim $2$ we have already proved that $-K_{X''}|_{C''}$ is a non trivial two-torsion divisor, so $i_{L''}|_{C''}$ is a Prym-canonical embedding.
  \end{proof}

\par\bigskip\noindent We observe that, since $i_{L''|_{C''}}:C''\rightarrow\mathbb{P}^{g-2}$ by definition of Prym-canonical map, then $g=g(C'')=a^2(q-1)+1.$

\noindent The image $x\in X$ of $-2K_{X''}\sim 4C_0$ by $i_{L''}$ is a singular point. There are not other possible singularities because $L''$ is an isomorphism outside $C_0$. We have found examples of surfaces in $\mathbb{P}^{a^2(q-1)}$ with Prym-canonical hyperplane sections birationally equivalent to non-rational ruled surfaces, for $a\geq 2$ and $q\geq3$.

\noindent If a general hyperplane section $C$ of $X$ is projectively normal, then, by Proposition \ref{sum sing}, the singularity $x$ has geometric genus equal to $q$.


\end{proof}

\par\bigskip We can construct an example of surface with Prym-canonical hyperplane sections birationally equivalent to an elliptic ruled surface $X''$.

\begin{exa}
  Let $\Gamma$ be an elliptic curve and let $X''=\mathbb{P}_{\Gamma}(\mathcal{O}_{\Gamma}\oplus \mathcal{O}_{\Gamma}(D))$ be a minimal ruled surface with base curve $\Gamma$, for $D\in \operatorname{Div}(\Gamma)$. If $Q_i$, for $i=1,2,3$, is a general point on $\Gamma$ and $\alpha,\beta\in\Gamma$ are two points such that $\alpha-\beta$ is a two torsion element of $\Gamma$, then we assume that $D=-Q_1-Q_2-Q_3-\alpha+\beta$. So $e=-\deg(D)=3$.

  \par\bigskip\noindent We consider the linear system $|3C_1|=|3C_0+3(Q_1+Q_2+Q_3+\alpha-\beta)\cdot f|$. It is easy to prove that this linear system is base-point free using \cite{FP}, Proposition $36$, so, by Bertini's Theorem, its general element $\mathcal{L}$ is smooth and, since $\mathcal{L}^2=9C_1^2=9e>0$, it is also irreducible.

  \noindent We call $f_i:=Q_i\cdot f$, for $i=1,2,3$. For any fibre $f$, we have that $\mathcal{L}\cdot f=3$, so we can fix the $9$ points $$Z:=\{x_{1,1},x_{1,2},x_{1,3},x_{2,1},x_{2,2},x_{2,3},x_{3,1},x_{3,2},x_{3,3}\}$$ of intersection between $\mathcal{L}$ and $f_i$, for $i=1,2,3$. Since $3C_1$ is disjoint from $C_0$, we can assume that $Z\cap C_0=\varnothing$. So we can consider the linear system $L''\subset|3C_1|$ on $X''$ with $Z$ as base locus. In particular, we suppose that every curve $C''\in L''$ simply passes through the $9$ points, so $\mathcal{L}$ is an element of $L''$. By \cite{FP}, Lemma $35$, we have that $h^0(\mathcal{O}_{X''}(3C_0+3(Q_1+Q_2+Q_3+\alpha-\beta)\cdot f))=19$, so $\dim L''\geq18-9=9$. Since $\mathcal{L}\in L''$ and smoothness is an open condition, then the general element $C''$ of $L''$ is smooth.

  \par\bigskip\noindent We know that $-K_{X''}\sim 2C_0-D\cdot f$ and $h^0(\mathcal{O}_{X''}(-K_{X''}))=4$ by \cite{FP}, Lemma $35$.

  \noindent We can show that $h^0(\mathcal{O}_{X''}(-K_{X''}-\mathcal{I}_{Z}))=0$. Indeed, if we suppose that there is an effective divisor $T\in |-K_{X''}-\mathcal{I}_{Z}|$, then $T\cdot C_0=2C_0^2-\deg(D)=2(-3)+3<0$, so $C_0$ is a fixed component of $T$. Moreover $T\cdot f=2$ but $T$ contains $3$ points of the fibres $f_i$, so it also contains $f_1,f_2,f_3$. Then $$|-K_{X''}-\mathcal{I}_{Z}|=|2C_0-D\cdot f-\mathcal{I}_{Z}|=C_0+f_1+f_2+f_3+|C_0+(\alpha-\beta)\cdot f|,$$ so $$\dim |-K_{X''}-\mathcal{I}_{Z}|=\dim |C_0+(\alpha-\beta)\cdot f|.$$

  \noindent By \cite{FP}, Lemma $35$, we have that $h^0(\mathcal{O}_{X''}(C_0+(\alpha-\beta)\cdot f))=h^0(\mathcal{O}_{\Gamma}(\alpha-\beta))+h^0(\mathcal{O}_{\Gamma}(\alpha-\beta-Q_1-Q_2-Q_3-\alpha+\beta))=0$. Hence $T$ effective does not exist.

  \par\bigskip\noindent It is not difficult to prove that $h^0(\mathcal{O}_{X''}(-2K_{X''}))=10$ and, since $-2K_{X''}\sim 4C_0+2(Q_1+Q_2+Q_3)\cdot f$ and $4C_0+2(Q_1+Q_2+Q_3)\cdot f$ contains $Z$ with multiplicity $2$, then also $h^0(\mathcal{O}_{X''}(-2K_{X''})\otimes\mathcal{I}_{Z})>0$.


  \par\bigskip\noindent Let $\phi: X'\rightarrow X''$ be the blowing up of $X''$ along the $9$ points defining $Z$. Let $E_{i,j}\in X'$ be the exceptional divisor of $x_{i,j}$, for $i,j\in\{1,2,3\}$, and let $\widetilde{f}_i$ be the strict transform of $f_i$, for $i=1,2,3$. With abuse of notation, we call $C_0:=\phi^*(C_0)$ (we remark that $\phi^*(C_0)$ is the strict transform of $C_0$ since $Z\cap C_0=\varnothing$), $\phi^*(\alpha\cdot f):=f_{\alpha}$ and $\phi^*(\beta\cdot f):=f_{\beta}$. Let $L'$ be such that $L''=\phi_*L'$. Then the strict transform $C'\in L'$ of a general $C''\in L''$ is of the form
  \begin{eqnarray*}
           C' =  \phi^*(C'')- E_{1,1}-...-E_{3,3} &\sim & 3C_0+3\sum_{i=1}^{3}\widetilde{f_i}+3f_{\alpha}-3f_{\beta}+\end{eqnarray*}\begin{eqnarray*}+2(E_{1,1}+...+E_{3,3}).
   \end{eqnarray*}

\bigskip\noindent Instead, using \cite{H}, Proposition $V.3.3$,  we obtain that \begin{eqnarray*}
                     -K_{X'}= \phi^*(-K_{X''})- E_{1,1}-...-E_{3,3}&\sim & 2C_0+\sum_{i=1}^{3}\widetilde{f_i}+f_{\alpha}-f_{\beta}
                   \end{eqnarray*}
\noindent and $$-2K_{X'}\sim 4C_0+2\widetilde{f_1}+2\widetilde{f_2}+2\widetilde{f_3}.$$

\par\bigskip\noindent It is clear that $h^0(\mathcal{O}_{X'}(-K_{X'}))=h^0(\mathcal{O}_{X''}(-K_{X''}-\mathcal{I}_{Z}))=0$ while $h^0(\mathcal{O}_{X'}(-2K_{X'}))=h^0(\mathcal{O}_{X''}(-2K_{X''}-\mathcal{I}_{Z}))>0$.

\par\bigskip Step by step, we can prove that the general hyperplane section $C'$ of $X'$ is a Prym-canonical embedded curve.

\par\bigskip\noindent $\mathrm{CLAIM \hspace{0.2cm}1:}$ We have that $\mathcal{O}_{C'}(-K_{X'})\ncong \mathcal{O}_{C'}$ and $\mathcal{O}_{C'}(-2K_{X'})\cong\mathcal{O}_{C'}$, where $C'\in L'$ is a general curve. In particular, $-K_{X'}|_{C'}$ is a non-zero two torsion divisor.

\begin{proof} The intersection $$C'\cdot( -K_{X'})=6C_0^2+6\sum_{i=1}^{3}C_0\cdot \widetilde{f_i}+3\sum_{i=1}^{3}(C_0\cdot \widetilde{f_i}+\widetilde{f_i}^2)+2[\widetilde{f_1}\cdot(E_{1,1}+E_{1,2}+E_{1,3})+$$$$\widetilde{f_2}\cdot(E_{2,1}+E_{2,2}+E_{2,3})+\widetilde{f_3}\cdot(E_{3,1}+E_{3,2}+E_{3,3})]=6(-3)+6(3)+3(3-9)+2(3+3+3)=0$$
\par\bigskip\noindent and clearly also $C'\cdot (-2K_{X'})=0$.

\noindent Since $-2K_{X'}$ is effective, then the antibicanonical divisor of $X'$ is contracted by $i_{L'}$, in particular $\mathcal{O}_{C'}(-2K_{X'})\cong \mathcal{O}_{C'}$.

\noindent On the contrary, we have that $h^0(\mathcal{O}_{X'}(-K_{X'}))=0$, so $-K_{X'}$ is not effective. As seen in Claim $2$ of Proposition \ref{example}, if we proved that $h^1(\mathcal{O}_{X'}(-K_{X'}-C'))=0$, then $\mathcal{O}_{C'}(-K_{X'})\ncong \mathcal{O}_{C'}$.

\par\bigskip\noindent Thus, by Serre Duality, it is clear that $h^1(\mathcal{O}_{X'}(-K_{X'}-C'))=h^1(\mathcal{O}_{X'}(2K_{X'}+C'))$. If we prove that $K_{X'}+C'$ is big and nef, then, by the Kawamata-Viehweg vanishing Theorem (see \cite{K} and \cite{V}), the first cohomology $h^1(\mathcal{O}_{X'}(2K_{X'}+C'))=0$.

\par\bigskip\noindent In our case, since $2(\alpha-\beta)\sim 0$, then $2f_{\alpha}-2f_{\beta}\sim 0$ and the divisor $$K_{X'}+C'\sim C_0+2\widetilde{f_1}+2\widetilde{f_2}+2\widetilde{f_3}+2E_{1,1}+...+2E_{3,3}.
$$

\noindent Since $K_{X'}+C'$ is written as sum of irreducible and effective curves, then, to prove that $K_{X'}+C'$ is nef, it is enough to prove that $(K_{X'}+C')\cdot \delta\geq0$, for any its irreducible component $\delta$.
With some simple computations we obtain that this is true and because we have strictly positive intersections between $K_{X'}+C'$ and its components, then $K_{X'}+C'$ is also big. So the claim is satisfied.\end{proof}

\par\bigskip\noindent It is not difficult to compute that $C'^2=18$, so, by the adjunction formula, the genus $g(C')=1+\frac{1}{2}(C'^2+K_{X'}\cdot C')=10$. We also observe that $C'$ is smooth because it is the strict transform of a general element $C''$ of $L''$, that is smooth. Since $-K_{X'}|_{C'}$ is a non-zero two torsion divisor as seen in Claim $1$, we have that $L'|_{C'}= |K_{C'}-K_{X'}|_{C'}|$ defines a Prym-canonical map $$\phi_{L'|_{C'}}:C'\dashrightarrow \mathbb{P}^8.$$

\par\bigskip\noindent $\mathrm{CLAIM \hspace{0.2cm}2:}$ The rational map $\phi_{L'|_{C'}}:C'\dashrightarrow \mathbb{P}^8$ is an embedding, for any general curve $C'\in L'$.

\begin{proof}
  First of all, we know that, if $L'|_{C'}$ has base points, then $C'$ is hyperelliptic by \cite{CDGK}, Lemma $2.1$. Moreover, since $C'\cdot \widetilde{f}=3$, where $\widetilde{f}$ is the pullback of a general fibre $f$ of $X''$, then $C'$ is also a covering $3:1$ of the elliptic curve $\Gamma$.
  This is not possible by Castelnuovo-Severi inequality otherwise we would have $10=g(C')\leq2\cdot0+3\cdot1+1\cdot 2=5$ (see \cite{A}).
Thus $L'|_{C'}$ is base-point free.

  \par\bigskip\noindent Thanks to \cite{CDGK}, Corollary $2.2$, we know that, if $C'$ is not bielliptic, then $L'|_{C'}$ is an embedding. Because $C'$ is a triple cover of $\Gamma$ as observed before, then $C'$ cannot be bielliptic again by Castelnuovo-Severi inequality otherwise we would have $10\leq2\cdot 1+3\cdot 1+1\cdot 2=7$. Thus the claim is proved.
\end{proof}

\par\bigskip\noindent At this point, since $L'|_{C'}$ is base-point free, it is clear that $L'$ is also base-point free.

\noindent Since the restriction $L'|_{C'}$ defines an embedding for each generic curve $C'\in L'$, then $\phi_{L'}$ is a birational map, generically $1:1$.

\par\bigskip\noindent Before we have showed that $\dim(L'')=\dim(L')\geq9$. From the exact sequence $$0\rightarrow \mathcal{O}_{X'}(C'-C')\rightarrow \mathcal{O}_{X'}(C')\rightarrow \mathcal{O}_{C'}(C')\rightarrow 0,$$ we conclude that $h^0(\mathcal{O}_{X'}(C'))\leq10$ since $\mathcal{O}_{X'}(C'-C')\cong \mathcal{O}_{X'}$ and $h^0(\mathcal{O}_{C'}(C'))=9$. So we have that $h^0(\mathcal{O}_{X'}(C'))=10$.

\par\bigskip\noindent Then $X'$ has hyperplane sections that are Prym-canonically embedded and, in particular, we have found a new surface $X=\phi_{L'}(X')\subset \mathbb{P}^9$ with Prym-canonical hyperplane sections. Since the antibicanonical divisor of $X'$ is connected, then its image $x\in X$ by $\phi_{L'}$ is a singular point. There are other possible rational double singularities on $X$ whose exceptional divisors on $X'$ do not intersect $-2K_{X'}$.

\noindent If a general hyperplane section $C$ of $X$ is projectively normal, then, by Proposition \ref{sum sing}, the geometric genus $p_g(x)$ is equal to $1$.
\end{exa}

\par\bigskip\begin{rem}
  We can compute how many moduli the couple $(X'',L'')$ of previous example depends on.

  \bigskip\noindent The choice of the elliptic curve $\Gamma$ depends on one parameter. In addition we fix a divisor $D=-Q_1-Q_2-Q_3-\alpha+\beta$ of degree $-3$, where $Q_i$ is a general point on $\Gamma$, for $i=1,2,3$, and $\alpha-\beta$ is a non-zero two torsion element of $\Gamma$.

  \noindent We know that there are only three non-zero two torsion points on $\Gamma$. Instead we observe that $|Q_1+Q_2+Q_3|$ is a linear system of dimension $2$, so the choice of $\mathcal{O}_{\Gamma}(D)$ depends on $3-2=1$ parameter.

   \noindent Moreover, every automorphism of $\Gamma$ lifts to an automorphism of $X''$ that means that, if $\phi:\Gamma\rightarrow \Gamma$ is an automorphism, then $X''=\mathbb{P}_{\Gamma}(\mathcal{O}_{\Gamma}\oplus\mathcal{O}_{\Gamma}(D))\cong \linebreak \mathbb{P}_{\Gamma}(\mathcal{O}_{\Gamma}\oplus\mathcal{O}_{\Gamma}(\phi^*(D)))$. The group $\operatorname{Aut}(\Gamma)$ has dimension $1$.

  \bigskip\noindent To construct the surface with Prym-canonical hyperplane sections of the previous example, we also fix a linear system $L''\subset|3C_1|=|3C_0-3D\cdot f|$ with $9$ simple base points. The linear system $|3C_1|$ depends on the parameters fixed before. Instead the $9$ simple base points are the points of intersection between a general element $\mathcal{L}\in |3C_1|$ and the three fibres $f_i:=Q_i\cdot f$, for $i=1,2,3$. The choice of the effective divisor in a linear system of the type $|Q_1+Q_2+Q_3|$ that defines the
  three fibres $f_1,f_2,f_3$ depends on $2$ parameters. In addition, as seen in the previous example, the nine points $\{x_{1,1},x_{1,2},x_{1,3},x_{2,1},x_{2,2},x_{2,3},x_{3,1},x_{3,2},x_{3,3}\}$ impose independent conditions on the linear system $L''$, so they depend on $9$ parameters.

  \par\bigskip\noindent The choice of the pair $(X'',L'')$ depends on $1+1-1+2+9=12$ parameters.

  \par\bigskip\noindent We know that $\dim(\operatorname{Aut}(\mathbb{P}^3))=15$. We can consider $X''$ as the blowing up of the vertex of the cone $C_{X''}$ on a plane cubic of $\mathbb{P}^3$. If $C_{\Gamma}$ is the base curve of $C_{X''}$, there are $\infty^8$ plane cubics isomorphic to $C_{\Gamma}$. Since we can choose the vertex among all the possible points of $\mathbb{P}^3$ obtaining always isomorphic cones, then there are $\infty^{(8+3)}=\infty^{11}$ isomorphic cones of the type of $C_{X''}$ in $\mathbb{P}^3$.

  \noindent Thus there are $\infty^4$ automorphism of $\mathbb{P}^3$ that fix $X''$ so, in conclusion, the couple $(X'',L'')$ depends on $12-4=8$ parameters.

  \par\bigskip\noindent Since the surface constructed in the previous example depends on $8$ moduli while a general Enriques surface depends on $10$ moduli, then the generic Enriques surface can degenerate to one of these surfaces since they depend on less parameters.
\end{rem}

\subsubsection{\textbf{The minimal model is a rational ruled surface}}
\par\bigskip\noindent We construct an example of surface with Prym-canonical hyperplane sections birationally equivalent to a rational ruled surface.

\begin{exa}
  Let $\Gamma$ be a rational smooth curve and let $X''=\mathbb{P}_{\Gamma}(\mathcal{O}_{\Gamma}\oplus\mathcal{O}_{\Gamma}(D))$ be a minimal ruled surface with base curve $\Gamma$, for $D\in \operatorname{Div}(\Gamma)$. We assume that $e=4$, so $\deg(D)=-4$. Hence $X''$ is a Hirzebruch surface $F_4$.

  \par\bigskip\noindent We know that $-K_{X''}\sim2C_0-(K_{\Gamma}+D)\cdot f$, where $\deg(-K_{\Gamma}-D)=2+4=6$. We put $$-K_{X''}=2C_0+2\sum_{i=1}^{3}F_i,$$ where $F_1,F_2$ and $F_3$ are distinct and fixed fibres. We also set $$W''=4C_0+3F+3\sum_{i=1}^{3}F_i\in |-2K_{X''}|,$$ where $F$ is a generic fibre distinct from $F_i$, for $i=1,2,3$.

  \par\bigskip\noindent We consider the linear system $|4C_1|=|4C_0-4D\cdot f|$ on $X''$. Every element in $|4C_1|$ intersects every fibre of $X''$ in $4$ points since $4C_1\cdot f=4$. In addition, we know that $h^0(\mathcal{O}_{X''}(4C_1))=45$ by \cite{FP}, Lemma $35$.

  \par\bigskip\noindent
$\mathrm{CLAIM \hspace{0.2cm}1:}$ There is a smooth curve $\mathcal{L}\in |4C_1|$ such that $\mathcal{L}$ is tangent to $F$, $F_1$, $F_2$ and $F_3$ respectively in two points.

\begin{proof} We can assume that $X''=F_4$ is the blowing up of the vertex of a cone in $\mathbb{P}^5$ on a rational normal curve of $\mathbb{P}^4$. It is clear that $C_0$ is the exceptional divisor associated with the vertex.

  \noindent Let $E=F+F_1+F_2+F_3$ be the curve intersection between the cone and a hyperplane $H_0$ of $\mathbb{P}^5$ passing through the vertex of the cone. With abuse of notation, the total transform of $E$ on $X''$ is $E=C_0+F+F_1+F_2+F_3$.

  \par\bigskip\noindent It is obvious that the linear systems $|2C_1|$ and $|3C_1|$ are base-point free. Then they respectively contain a general quadric $Q$ and a general cubic $C$, that are smooth by Bertini's Theorem.

  \noindent It is clear that $2Q$ intersects every fibre in two points with multiplicity $2$. We call $\{x_j,x_{1,j},x_{2,j},x_{3,j}\}$, for $j=1,2$, the intersections points between $2Q$ and $F+F_1+F_2+F_3$.

  \noindent Let us consider a pencil $\mathcal{P}$ generated by $2Q$ and $E+C$. By Bertini's Theorem, its curves may have singular points only on the base locus of the pencil. At this point we observe that $2Q\cdot C\sim 2(2C_1)\cdot (3C_1)=12\cdot 4=48$ since $C_1$ is a rational normal curve of degree $4$. These $48$ points are base points for the pencil, different from $\{x_j,x_{1,j},x_{2,j},x_{3,j}\}$, with $j=1,2$, for the generality of $C$. Now $E+C$ has only $16$ singular points since $E$ has only $4$ singular points on $C_0$, while $C$ is smooth and disjoint from $C_0$ for its generality and $E\cdot C\sim (C_0+4F)\cdot 3C_1=12$. Since $Q$ is also disjoint from $C_0$ and it is general, then these $16$ points are different from the $48$ base points. Then $E+C$ is smooth in the $48$ base points. The same is true for $2Q$. Hence also a general divisor $\mathcal{L}$ in the pencil $\mathcal{P}$ is smooth in the $48$ base points.

  \noindent Instead $2Q|_{E}=\{x_1,x_2,x_{1,1},x_{1,2},x_{2,1},x_{2,2},x_{3,1},x_{3,2}\}$. These are other $8$ base points for $\mathcal{P}$. Since $2Q$ passes through $\{x_j,x_{1,j},x_{2,j},x_{3,j}\}$, for $j=1,2$, with multiplicity $2$ and since $E+C$ simply passes through the eight points ($E$ contains the fibres $F,F_1,F_2,F_3$ and $C$ does not contain these $8$ points), then a general curve $\mathcal{L}$ is smooth in these $8$ points and, in particular, it is tangent to $F,F_1,F_2,F_3$ in $\{x_1,x_2,x_{1,1},x_{1,2},x_{2,1},x_{2,2},x_{3,1},x_{3,2}\}$.

  \noindent Finally, we observe that $2Q\sim 2(2C_1)=4C_1$ and similarly, we have $C\sim 3C_1$ and $E\sim C_0+4F$, so $E+C\sim 4C_1$. Then we have found a smooth curve $\mathcal{L}\in |4C_1|$ tangent to $F$ in $x_1$ and $x_2$ and tangent to $F_i$ in $x_{i,1}$ and $x_{i,2}$, for $i=1,2,3$. \end{proof}

  \par\bigskip In the following figure, we analyze what happens blowing up all the intersection points between $\mathcal{L}$ and $W''$, also infinitely near. We observe that, since $\mathcal{L}\sim 4C_1$ and $4C_1$ is disjoint from $C_0$, then $\mathcal{L}$ does not intersect $C_0$. With abuse of notation, we will call the strict transforms of $C_0$, $F_i$ and $\mathcal{L}$ with the same names.

 \noindent In the figure, we only focus on $F_1$, it is the same for $F_2,F_3$ and $F$.
 \bigskip

  \begin{itemize}
    \item{STEP 1} We blow up the intersection points $x_{i,1}$ and $x_{i,2}$ on $X''$;
   \bigskip \item{STEP 2} In $X''_1:=Bl_{x_{1,1},x_{1,2},x_{2,1},x_{2,2},x_{3,1},x_{3,2}}(X'')$, the curve $\mathcal{L}$ simply passes through the infinitely near base points $y_{i,1}$ and $y_{i,2}$, for $i=1,2,3$. We also blow up these six points;
   \bigskip \item{STEP 3} Again $\mathcal{L}$ intersects the exceptional divisors $E_{i,1}$ and $E_{i,2}$ respectively in $z_{i,1}$ and $z_{i,2}$ on $X''_2:=Bl_{y_{1,1},y_{1,2},y_{2,1},y_{2,2},y_{3,1},y_{3,2}}(X''_1)$, for $i=1,2,3$. We obtain $Y=Bl_{z_{1,1},z_{1,2},z_{2,1},z_{2,2},z_{3,1},z_{3,2}}(X''_2)$ blowing up these other six points.
  \end{itemize}

 \par\bigskip
  \begin{figure}[h]
\centering
\includegraphics [scale=0.6]{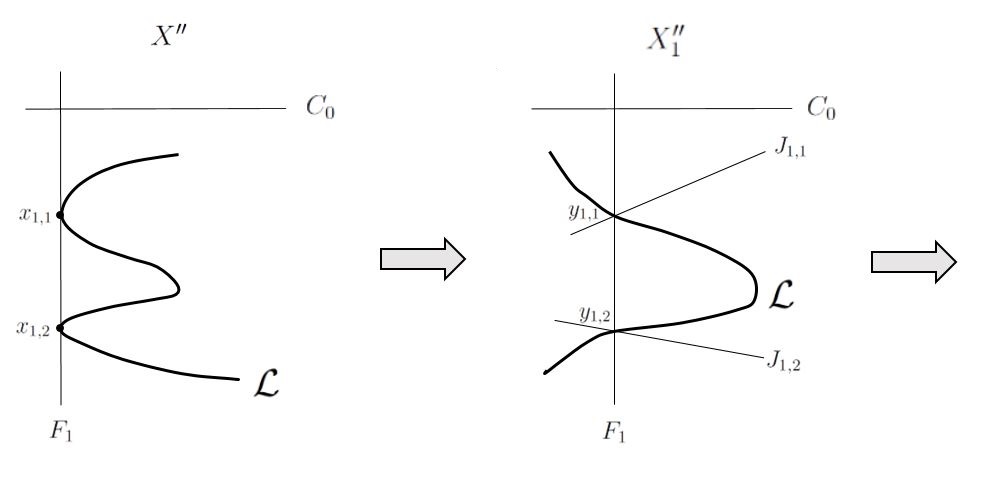}
\end{figure}

\begin{figure}[h]
\includegraphics [scale=0.6]{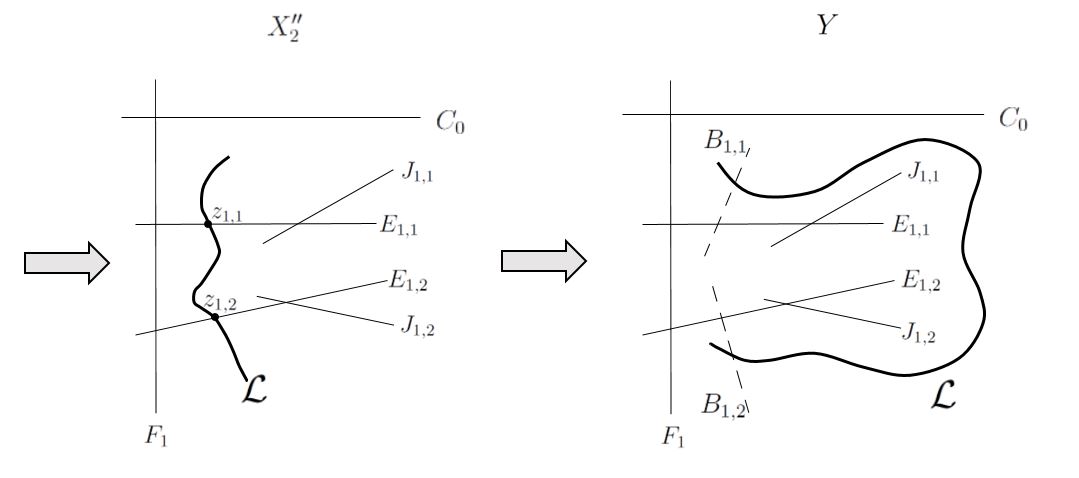}
\label{fig:1}
\end{figure}

\newpage
  \noindent With the same techniques as before, we also blow up $\{x_1,x_2,y_1,y_2,z_1,z_2\}\in \mathcal{L}\cap F$ (as seen in the previous figure, they are infinitely near points). We define $$X':=Bl_{z_1,z_2}(Bl_{y_1,y_2}(Bl_{x_1,x_2}(Y))).$$ In $X'$, there are no intersection points between $\mathcal{L}$ and $-2K_{X'}$.

  \par\bigskip After observing how the blowing up works, we consider $L''\subset|4C_1|$ as the linear system of the curves of $|4C_1|$ simply passing through $$Z:=\{x_1,x_2,y_1,y_2,z_1,z_2,x_{i,1},x_{i,2},y_{i,1},y_{i,2},z_{i,1},z_{i,2}\}, \hspace{0.2cm} for \hspace{0.2cm} i=1,2,3.$$ Then $\dim L''\geq 44-6\cdot 4=20$. It is clear that $\mathcal{L}$ is an element of $L''$ and since smoothness is an open condition, then the general element $C''$ of $L''$ is smooth.

  \par\bigskip\noindent With the same notation as before, we can obtain that $$-K_{Y}=2C_0+\sum_{i=1}^{3}(2F_i+J_{i,1}+J_{i,2}+2E_{i,1}+2E_{i,2}+B_{i,1}+B_{i,2})$$ while $$W''_Y=4C_0+3F+\sum_{i=1}^{3}(3F_i+J_{i,1}+J_{i,2}+2E_{i,1}+2E_{i,2})\in |-2K_Y|.$$

  \par\bigskip\noindent At this point, we use the fact that, if $M$ is an effective divisor and $N_i$ are irreducible divisors, if $M\cdot N_1<0$, $(M-N_1)\cdot N_2<0$, and so on, the $\sum N_i$ is a partial fixed part of $|M|$. Thus, inductively, one can verify that all $2C_0+\sum_{i=1}^{3}(2F_i+J_{i,1}+J_{i,2}+2E_{i,1}+2E_{i,2}+B_{i,1}+B_{i,2})$ is a fixed component of $|-K_{Y}|$, so this is the only one effective curve in its linear system. In addition, the part $4C_0+\sum_{i=1}^{3}(3F_i+J_{i,1}+J_{i,2}+2E_{i,1}+2E_{i,2})$ is the fixed part of $|-2K_{Y}|$ while its variable part is $3F$.

  \par\bigskip\noindent Similarly to before, we can compute that
   $$-K_{X'}\sim2C_0+\sum_{i=1}^{3}(2F_i+J_{i,1}+J_{i,2}+2E_{i,1}+2E_{i,2}+B_{i,1}+B_{i,2})+$$$$-J_1-J_2-2E_1-2E_2-3B_1-3B_2.$$

\bigskip\noindent This is clearly not effective. Instead $$W'=4C_0+3F+J_1+J_2+2E_1+2E_2+\sum_{i=1}^{3}(3F_i+J_{i,1}+J_{i,2}+2E_{i,1}+2E_{i,2})\in |-2K_{X'}|$$ is effective. In addition, all $4C_0+3F+J_1+J_2+2E_1+2E_2+\sum_{i=1}^{3}3F_i+J_{i,1}+J_{i,2}+2E_{i,1}+2E_{i,2}$ is a fixed component of $|-2K_{X'}|$ and consequently it is the only one effective divisor in its linear system.

 \par\bigskip\noindent Since all the divisors on $\Gamma$ of the same degree are linearly equivalent, then we observe that $-4D\cdot f\sim 4F+4\sum_{i=1}^{3}F_i$, so we can assume that $$C''\sim4C_0+4F+4\sum_{i=1}^{3}F_i,$$ where $C''$ is a general element in $L''$. Then, its strict transform $C'$ on $X'$ is linearly equivalent to $$C'\sim 4C_0+4F+\sum_{i=1}^{2}(3J_j+6E_j+5B_j)+$$$$+\sum_{i=1}^{3}(4F_i+3J_{i,1}+3J_{i,2}+6E_{i,1}+6E_{i,2}+5B_{i,1}+5B_{i,2}).$$ If $\phi:X'\rightarrow X''$ is the blowing up of $X''$ along the points of $Z$, let $L'$ be such that $L''=\phi_*L'$, with $C'$ a general element.

\par\bigskip Step by step, we can prove that a general hyperplane section $C'$ of $X'$ is a Prym-canonical embedded curve.

\par\bigskip\noindent $\mathrm{CLAIM \hspace{0.2cm}2:}$ We have that $\mathcal{O}_{C'}(-K_{X'})\ncong \mathcal{O}_{C'}$ while $\mathcal{O}_{C'}(-2K_{X'})\cong\mathcal{O}_{C'}$. In particular, $-K_{X'}|_{C'}$ is a non-zero two torsion divisor.

\begin{proof} It is easy to compute that $C'\cdot W'=0$. Since $W'$ is effective, then it is contracted by the map defined by $L'$, in particular $\mathcal{O}_{C'}(-2K_{X'})\cong \mathcal{O}_{C'}$.

\noindent It is clear that also $C'\cdot (-K_{X'})=0$ but this time we have that $h^0(\mathcal{O}_{X'}(-K_{X'}))=0$. As seen in Claim $2$ of Proposition \ref{example}, it is sufficient to show that \linebreak$h^1(\mathcal{O}_{X'}(-K_{X'}-C'))=0$ to prove that $\mathcal{O}_{C'}(-K_{X'})\ncong \mathcal{O}_{C'}$.

\noindent Using Serre Duality and the Kawamata-Viehweg vanishing Theorem (see \cite{K} and \cite{V}), if we prove that $K_{X'}+C'$ is big and nef, then the claim is satisfied.
\noindent Now $$K_{X'}+C'=2C_0+\sum_{i=1}^{3}(2F_i+2J_{i,1}+2J_{i,2}+4E_{i,1}+4E_{i,2}+4B_{i,1}+4B_{i,2})+$$$$+4F+4J_1+4J_2+8E_1+8E_2+8B_1+8B_2.$$

\noindent Since $K_{X'}+C'$ is written as sum of irreducible and effective curves, then, to prove that $K_{X'}+C'$ is nef, it is enough to prove that $(K_{X'}+C')\cdot \delta\geq0$, for any its irreducible component $\delta$. It is possible to compute that this is true. Because we have strictly positive intersections between $K_{X'}+C'$ and its components, then $K_{X'}+C'$ is also big. So the claim is satisfied.\end{proof}

\par\bigskip\noindent It is not difficult to compute that $C'^2=40$, so, by the adjunction formula, the genus $g(C')=1+\frac{1}{2}(C'^2)=21$. We also observe that $C'$ is smooth because it is the strict transform of a general element $C''$ of $L''$, that is smooth. Since $-K_{X'}|_{C'}$ is a non-zero two torsion divisor as seen in Claim $2$, we have that $L'|_{C'}= |K_{C'}-K_{X'}|_{C'}|$ defines a Prym-canonical map $$\phi_{L'|_{C'}}:C'\dashrightarrow \mathbb{P}^{19}.$$

\par\bigskip\noindent We have already observed that $\dim(L'')=\dim(L')\geq20$. From the exact sequence $$0\rightarrow \mathcal{O}_{X'}(C'-C')\rightarrow \mathcal{O}_{X'}(C')\rightarrow \mathcal{O}_{C'}(C')\rightarrow 0,$$ we conclude that $h^0(\mathcal{O}_{X'}(C'))\leq20$ since $\mathcal{O}_{X'}(C'-C')\cong \mathcal{O}_{X'}$ and $h^0(\mathcal{O}_{C'}(C'))=19$. So we have that $h^0(\mathcal{O}_{X'}(C'))=21$ and $\phi_{L'}(X')\subseteq\mathbb{P}^{20}.$

\par\bigskip\noindent $\mathrm{CLAIM \hspace{0.2cm}3:}$ The rational map $\phi_{L'|_{C'}}:C'\dashrightarrow \mathbb{P}^{19}$ is an embedding, for any general curve $C'\in L'$.

\begin{proof}
  First of all, we know that, if the Prym-canonical system $L'|_{C'}$ has base points, then $C'$ is hyperelliptic by \cite{CDGK}, Lemma $2.1$.

\noindent Let us suppose that $C'$ is hyperelliptic. Since we know that $C'\in X'$ is isomorphic to $C''\sim 4C_1$ on $X''$, then $C''$ is also hyperelliptic.

\noindent The self-intersection $C''^2=(4C_1)^2=64$. Furthermore $C''$ is nef since $C''\sim 4C_0+4F+4\sum_{i=1}^{3}F_i$ and $(4C_0+4F+4\sum_{i=1}^{3}F_i)\cdot C_0=0$ and $(4C_0+4F+4\sum_{i=1}^{3}F_i)\cdot F=(4C_0+4F+4\sum_{i=1}^{3}F_i)\cdot F_i=4$.
Moreover, by \cite{H}, Proposition $IV.5.2$, we have that $|K_{C''}|$ is not very ample, precisely it does not separate any pair of points $p$ and $q$ such that $p+q$ is a member of the $g^1_2$ on $C''$. By the adjunction formula, we also have that $|K _{X''} + C''|$ does not
separate such $p$ and $q$.


\noindent By \cite{R}, Theorem $1.$, there exists an effective divisor $E$ on $X''$ passing through $p$ and $q$ such that $C''\cdot E<4$. Since $C''\sim 4C_1$, then $C''\cdot E=0$. This is not possible and thus $E$ cannot exist. We exclude the case $C''$ hyperelliptic and hence $L'|_{C'}$ is base-point free.

  \par\bigskip\noindent Furthermore, we can prove that $L'|_{C'}$ defines a birational map. Indeed, if this did not happen, we would have $C'$ bielliptic and the image of $X'$ via the map associated with $L'$ would be a surface in $\mathbb{P}^{20}$ with elliptic sections (see \cite{CDGK}, Corollary $2.2$). Since $20>9$, then the surface image in $\mathbb{P}^{20}$ could not be a Del Pezzo surface but it would be an elliptic cone. Anyway $X'$ is a rational surface, so it cannot cover an elliptic cone. Then $L'|_{C'}$ defines a birational map.

 \par\bigskip\noindent More precisely, we can also show that $L'|_{C'}$ defines an embedding, for any general $C'\in L'$. By \cite{CDGK}, Lemma $2.1$, we know that $L'|_{C'}$ does not separate $p$ and $q$ (possibly infinitely near) if and only if $C'$ has a $g^1_4$ and $-K_{X'}|_{C'}\sim \mathcal{O}_{C'}(p+q-x-y)$, where $2(p+q)$ and $2(x+y)$ are members of the $g^1_4$.

 \par\bigskip\noindent We know that $C'\cong C''$ and $C''\sim 4C_1$ has a $g^1_4$ defined by the fibres of the ruled surface $X''$. This is the only one. Indeed, if $C'$ had two $g^1_4$, then there would be a map $\psi:C'\rightarrow \mathbb{P}^1\times\mathbb{P}^1$. If $\psi$ was a birational map, the image curve would be of the type $(4,4)$ on $\mathbb{P}^1\times\mathbb{P}^1$. Then its geometric genus would be at most $(4-1)(4-1)=9$. Since $C'$ has genus $21$, this case is excluded. Thus $\psi$ would be a map $2:1$ on a curve $D$. The image curve $D$ would be a curve of type $(2,2)$ on $\mathbb{P}^1\times\mathbb{P}^1$, so its geometric genus would be $g(D)\leq1$. Since $C'$ is non-hyperelliptic as seen before, then $g(D)=1$ and $C'$ is bielliptic. Then $C'$ admits a singular correspondence. By Corollary $2.2$ of \cite{Cil-VdV}, the map determined by the linear system $|C'|$ is not birational, in particular it is $2:1$ on a surface with elliptic sections. We have already excluded this possibility, so $C''$ has only one $g^1_4$.

 \par\bigskip\noindent It is clear that $C''\cap F\sim C''\cap F_1\sim C''\cap F_2\sim C''\cap F_3$ and we know that $C''$ is tangent in two points to this four fibres. So we have four pairs of points $(p,q)\in F$, $(x,y)\in F_1$, $(z,w)\in F_2$ and $(a,b)\in F_3$ such that $2(p+q)\sim 2(x+y)$ and so on for all the possible cases. Since $C'$ is the strict transform of $C''$, it has the same characteristics of $C''$ and, after the blowing up, the four pairs of points that satisfy this property are the intersection points between $C'$ and $B_{i,j}$ and $C'$ and $B_j$, for $i=1,2,3$ and $j=1,2$. Now, with abuse of notation and using the expression of $-K_{X'}$ seen before, we have that $$-K_{X'}|_{C'}=\sum_{i=1}^{3}(B_{i,1}+B_{i,2})|_{C'}-(3B_1+3B_2)|_{C'}=$$$$=x+y+w+z+a+b-3p-3q\sim x+y-w-z+a+b-p-q.$$

 \bigskip\noindent At this point, we observe that $$x+y-w-z+a+b-p-q\nsim x+y-p-q$$ otherwise, if $a+b-w-z\sim 0,$ then $C'$ would have a $g^1_2$. Hence $L'|_{C'}$ separate each pair of points and it defines an embedding.
\end{proof}

\par\bigskip\noindent At this point, since $L'|_{C'}$ is base-point free, it is clear that $L'$ is also base-point free.

\noindent Since the restriction $L'|_{C'}$ defines an embedding for each generic curve $C'\in L'$, then $\phi_{L'}$ is a birational map, generically $1:1$.

 \noindent Then $X'$ has hyperplane sections that are Prym-canonically embedded. In particular, $\phi_{L'}(X')$ is a surface with Prym-canonical hyperplane sections.

\par\bigskip\noindent We have found a new surface $X=\phi_{L'}(X')\subset \mathbb{P}^{20}$ with Prym-canonical hyperplane sections. Since $W'$ is connected, then the image $x\in X$ of $W'$ is a rational singular point (see Proposition \ref{sum sing}). There are other possible rational double singularities on $X$ whose exceptional divisors on $X'$ do not intersect $-2K_{X'}$.
\end{exa}

\subsection{More surfaces with Prym-canonical hyperplane sections birationally equivalent to $\mathbb P^2$}
\par\bigskip\noindent We construct a new example of such surface whose minimal model is $X''=\mathbb{P}^2$.

\begin{exa}\label{Case 2}
  Let $X''=\mathbb{P}^2$ be such that $-2K_{X''}$ is an irreducible sextic with $10$ nodes $\{x_1,...,x_{10}\}$. Let $L''$ be a linear system of curves of degree $18$ with base points $\{x_1,...,x_{10}\}\in X''$ of multiplicity respectively $r_i=4$, for $i=1,2,3$, and $r_i=6$, for $i=4,...,10$. Let $X'=Bl_{\{x_1,...,x_{10}\}}(\mathbb{P}^2)$ be the blowing up of $X''$ along the base points of $L''$ and let $L'$ be the strict transform of $L''$. We observe that the anticanonical divisor $-K_{X'}$ is not effective.

 \noindent Let $$C'\sim 18 l-4\sum_{i=1}^{3}E_i-6\sum_{i=4}^{10}E_i$$ be a general curve in $L'$, where $E_i$ is the exceptional divisor associated with $x_i$, for $i=1,...,10$. It is obvious that $$\deg(C'|_{C'})=18^2-3\cdot 16-7\cdot36=24.$$ We have that $h^0(\mathcal{O}_{X'}(C'))\geq\binom{20}{2}-4\frac{4\cdot5}{2}-7\frac{6\cdot 7}{2}=13$, so $\phi_{L'}(X')=X\subset \mathbb{P}^r$, for $r\geq12$.

 \par\bigskip\noindent Since $-2K_{X'}\sim J=6l-2\sum_{i=1}^{10}E_i$ is effective and $C'\cdot (-2K_{X'})=0$ by construction, then $\mathcal{O}_{C'}(-2K_{X'})\cong \mathcal{O}_{C'}$. So $L'$ contracts $J$ in a single point since $J$ is irreducible. Moreover, since $J$ is also rational, then $\phi_{L'}(J)$ is a rational singularity of multiplicity $4$ because the fundamental cycle $Z_0=J$ is such that $Z_0^2=J^2=-4$.

 \par\bigskip\noindent $\mathrm{CLAIM \hspace{0.2cm}1:}$ The dimension $\dim(L')=12$, so $\phi_{L'}(X')=X\subseteq\mathbb{P}^{12}$.

 \begin{proof} We can consider the following exact sequence, already tensored with $\mathcal{O}_{X'}(C')$: \begin{equation}\label{P2 1}
                                                   0\rightarrow\mathcal{O}_{X'}(C'-J)\rightarrow\mathcal{O}_{X'}(C')\rightarrow\mathcal{O}_{J}(C')\rightarrow0.
                                                 \end{equation}

 \noindent Since $\mathcal{O}_{C'}(-2K_{X'})\cong \mathcal{O}_{C'}$ and $J\in |-2K_{X'}|$, then $\mathcal{O}_{J}(C')\cong \mathcal{O}_{J}\cong \mathcal{O}_{\mathbb{P}^1}$ since $J$ is rational. Thus we can rewrite (\ref{P2 1}) as \begin{equation}\label{P2 2}
                                                   0\rightarrow\mathcal{O}_{X'}(12l-2\sum_{i=1}^{3}E_i-4\sum_{i=4}^{10}E_i)\rightarrow\mathcal{O}_{X'}(18l-4\sum_{i=1}^{3}E_i-6\sum_{i=4}^{10}E_i)\rightarrow\mathcal{O}_{\mathbb{P}^1}\rightarrow0.
                                                 \end{equation}

\noindent Similarly we obtain that

$$0\rightarrow\mathcal{O}_{X'}(6l-2\sum_{i=4}^{10}E_i)\rightarrow\mathcal{O}_{X'}(12l-2\sum_{i=1}^{3}E_i-4\sum_{i=4}^{10}E_i)\rightarrow$$
\begin{equation}\label{P2 3}
 \rightarrow\mathcal{O}_{J}(12l-2\sum_{i=1}^{3}E_i-4\sum_{i=4}^{10}E_i)\rightarrow0.
                                                              \end{equation}

\bigskip\noindent It is possible to choose a quintuple of points among the $10$ nodes $\{x_1,...,x_{10}\}$ of $J$ such that three of these points are not aligned, then an irreducible conic passing through this quintuple of points exists. Up to renaming the nodes of $J$, we suppose that a conic passing through $\{x_4,...,x_8\}$ exists.

\noindent  So let us consider the following exact sequences:

\begin{equation}\label{P2 3bis}
 0\rightarrow\mathcal{O}_{X'}(4l-\sum_{i=4}^{8}E_i-2\sum_{i=9}^{10}E_i)\rightarrow\mathcal{O}_{X'}(6l-2\sum_{i=4}^{10}E_i)\rightarrow\mathcal{O}_{2l-\sum_{i=4}^{8}E_i}(6l-2\sum_{i=4}^{8}E_i)\rightarrow0;
                                                              \end{equation}

$$0\rightarrow\mathcal{O}_{X'}(3l-\sum_{i=4}^{10}E_i)\rightarrow\mathcal{O}_{X'}(4l-\sum_{i=4}^{8}E_i-2\sum_{i=9}^{10}E_i)\rightarrow$$
\begin{equation}\label{P2 3bisbis}
 \rightarrow\mathcal{O}_{l-E_9-E_{10}}(4l-\sum_{i=4}^{8}E_i-2\sum_{i=9}^{10}E_i)\rightarrow0.
                                                              \end{equation}

 \bigskip\noindent It obvious that $h^0(\mathcal{O}_{X'}(3l-\sum_{i=4}^{10}E_i))=\binom{5}{2}-7=3$. Since it is an effective divisor on $X'$ and it has the excepted dimension, then $h^1(\mathcal{O}_{X'}(3l-\sum_{i=4}^{10}E_i))=0$.

 \noindent Because $l-E_9-E_{10}$ is rational and $(l-E_9-E_{10})\cdot(4l-\sum_{i=4}^{8}E_i-2\sum_{i=9}^{10}E_i)=0$, then $h^0(\mathcal{O}_{l-E_9-E_{10}}(4l-\sum_{i=4}^{8}E_i-2\sum_{i=9}^{10}E_i))=1$ and $h^1(\mathcal{O}_{l-E_9-E_{10}}(4l-\sum_{i=4}^{8}E_i-2\sum_{i=9}^{10}E_i))=0$.

 \noindent From the exact sequence (\ref{P2 3bisbis}), we conclude that $h^0(\mathcal{O}_{X'}(4l-\sum_{i=4}^{8}E_i-2\sum_{i=9}^{10}E_i))=3+1=4$ and $h^1(\mathcal{O}_{X'}(4l-\sum_{i=4}^{8}E_i-2\sum_{i=9}^{10}E_i))=0$.

\par\bigskip\noindent Using the Riemann-Roch Theorem, since $J$ and $2l-\sum_{i=4}^{8}E_i$ are rational, we obtain that $$h^0(\mathcal{O}_{J}(12l-2\sum_{i=1}^{3}E_i-4\sum_{i=4}^{10}E_i))=4+1=5$$ and $$h^0(\mathcal{O}_{2l-\sum_{i=4}^{8}E_i}(6l-2\sum_{i=4}^{10}E_i))=2+1=3.$$

\par\bigskip\noindent  From the exact sequence (\ref{P2 3bis}), we conclude that $h^0(\mathcal{O}_{X'}(6l-2\sum_{i=4}^{10}E_i))=7$ and $h^1(\mathcal{O}_{X'}(6l-2\sum_{i=4}^{10}E_i))=0$. Again, from the exact sequence (\ref{P2 3}), we have that $h^0(\mathcal{O}_{X'}(12l-2\sum_{i=1}^{3}E_i-4\sum_{i=4}^{10}E_i))=12$ and $h^1(\mathcal{O}_{X'}(12l-2\sum_{i=1}^{3}E_i-4\sum_{i=4}^{10}E_i))=0$.
\noindent Finally, from the exact sequence (\ref{P2 2}), we obtain that $$h^0(\mathcal{O}_{X'}(18l-4\sum_{i=1}^{3}E_i-6\sum_{i=4}^{10}
E_i))=12+1=13.$$ Then the claim is proved. \end{proof}

\par\bigskip Step by step we want to show that a general hyperplane section of $X'$ is a Prym-canonical embedded curve.
 \par\bigskip\noindent $\mathrm{CLAIM \hspace{0.2cm}2:}$ We prove that $\mathcal{O}_{C'}(-K_{X'})\ncong \mathcal{O}_{C'}$.

 \begin{proof} If we show that $h^1(\mathcal{O}_{X'}(-K_{X'}-C'))=0$, then, using the long exact sequence associated with $$0\rightarrow \mathcal{O}_{X'}(-K_{X'}-C')\rightarrow \mathcal{O}_{X'}(-K_{X'})\rightarrow \mathcal{O}_{C'}(-K_{X'})\rightarrow 0$$ and observing that $-K_{X'}$ is not effective, we have that $h^0(\mathcal{O}_{C'}(-K_{X'}))=0$, thus $\mathcal{O}_{C'}(-K_{X'})\ncong \mathcal{O}_{C'}$.

\noindent Since $$-K_{X'}-C'\sim -15l+3\sum_{i=1}^{3}E_i+5\sum_{i=4}^{10}E_i,$$ then it is not effective and $h^0(\mathcal{O}_{X'}(-K_{X'}-C'))=0$. By Serre Duality, we have that $h^2(\mathcal{O}_{X'}(-K_{X'}-C'))=h^0(\mathcal{O}_{X'}(2K_{X'}+C'))=h^0(\mathcal{O}_{X'}(12l-2\sum_{i=1}^{3}E_i-4\sum_{i=4}^{10}E_i))=12$
as proved in Claim $1$.

\noindent Using the Riemann-Roch Theorem, we conclude that $-h^1(\mathcal{O}_{X'}(-K_{X'}-C'))+\linebreak +h^2(\mathcal{O}_{X'}(-K_{X'}-C'))=-h^1(\mathcal{O}_{X'}(-K_{X'}-C'))+12=\frac{1}{2}(-15l+3\sum_{i=1}^{3}E_i+5\sum_{i=4}^{10}E_i) \cdot(-12l+2\sum_{i=1}^{3}E_i+4\sum_{i=4}^{10}E_i)+1-0=12$, so $h^1(\mathcal{O}_{X'}(-K_{X'}-C'))=0$ and the claim is proved.\end{proof}

\par\bigskip\noindent $\mathrm{CLAIM \hspace{0.2cm}3:}$ There are irreducible curves of degree $18$ with exactly $3$ quadruple points and $7$ points of multiplicity six in the ten nodes of $J$.

\begin{proof}
We observe that curves of the type $J+D$, with $D\in|12l-2\sum_{i=1}^{3}E_i-4\sum_{i=4}^{10}E_i|$ and $J$ fixed part, are contained in $L'=|18l-4\sum_{i=1}^{3}E_i-6\sum_{i=4}^{10}E_i|$.

\noindent As proved in Claim $1$, we have that $\dim|18l-4\sum_{i=1}^{3}E_i-6\sum_{i=4}^{10}E_i|=12$ while $\dim |12l-2\sum_{i=1}^{3}E_i-4\sum_{i=4}^{10}E_i|=11$, so the reducible curves $J+D$ do not fill up all the linear system of the curves of degree $18$. As consequence of Bertini's Theorem (see \cite{Ak}, pag. $1$), the generic curve of $L'$ is irreducible (indeed the curves of the linear system $L'$ with fixed part $J$ define a sublinear system and moreover the sublinear system is not composed by a pencil, even more so the linear system $L'$).

\noindent Also curves of the type $2J+F$, with $F\in |6l-2\sum_{i=4}^{10}E_i|$ and $2J$ fixed part, are contained in $L'$. Since these special curves of $L'$ have exactly quadruple points in three of the $10$ nodes of $J$ and points of multiplicity $6$ in seven of the $10$ nodes of $J$, then the generic curves of the linear system $L'$ have the same property. Thus irreducible curves of degree $18$ with exactly quadruple points in three of the $10$ nodes of $J$ and points of multiplicity $6$ in the remaining nodes of $J$ exist. \end{proof}

\par\bigskip\noindent Therefore the arithmetic genus, that is equal to the geometric genus of $C$, is $$g(C')=\frac{17\cdot 16}{2}-3\frac{4\cdot 3}{2}-7\frac{6\cdot 5}{2}=13$$ by the Pl\"{u}cker Formula.

\par\bigskip\noindent It remains to show that $L'$ defines an embedding outside the contracted curve $J$.

\par\bigskip\noindent $\mathrm{CLAIM \hspace{0.2cm}4:}$ The linear system $L'$ is base-point free.

\begin{proof}
 Let $\overline{X'}=Bl_{x_{11}}(X')$, where $x_{11}$ is a point of a general $C'\in L'$. If $\overline{L'}=|18l-4\sum_{i=1}^{3}E_i-6\sum_{i=4}^{10}E_i-E_{11}|$, for $E_{11}$ the exceptional divisor associated with $x_{11}$, then $L'$ is base point free if and only if $$\dim(
 \overline{L'})=\dim(L')-1,$$ for any point $x_{11}\in C'$, for a general $C'\in L'$.

\par\bigskip\noindent We have already proved that $\dim(L')=12$, instead $\dim(\overline{L'})\geq\binom{20}{2}-3\frac{4\cdot 5}{2}-7\frac{6\cdot 7}{2}-1-1=11$. We observe that, since $x_{11}\in C'$ and $C'$ and $J$ are disjoint by assumptions, then $\mathcal{O}_{\overline{J}}(\overline{C'})\cong \mathcal{O}_{\mathbb{P}^1}$, where $\overline{C'}$ is a general curve in $\overline{L'}$ and $\overline{J}$ is the strict transform of $J$ on $\overline{X'}$.

 \noindent Similarly to the exact sequences (\ref{P2 2}), (\ref{P2 3}), we have the following:
 \begin{equation}\label{P2 4}
   0\rightarrow\mathcal{O}_{\overline{X'}}(12l-2\sum_{i=1}^{3}E_i-4\sum_{i=4}^{10}E_i-E_{11})\rightarrow\mathcal{O}_{\overline{X'}}(18l-4\sum_{i=1}^{3}E_i-6\sum_{i=4}^{10}E_i-E_{11})\rightarrow\mathcal{O}_{\mathbb{P}^1}\rightarrow0
 \end{equation}
$$
 0\rightarrow\mathcal{O}_{\overline{X'}}(6l-2\sum_{i=4}^{10}E_i-E_{11})\rightarrow\mathcal{O}_{\overline{X'}}(12l-2\sum_{i=1}^{3}E_i-4\sum_{i=4}^{10}E_i-E_{11})\rightarrow$$
\begin{equation}\label{P2 5} \rightarrow\mathcal{O}_{\overline{J}}(12l-2\sum_{i=1}^{3}E_i-4\sum_{i=4}^{10}E_i-E_{11})\rightarrow0.
                                                              \end{equation}

\noindent As in Claim $1$, we suppose that an irreducible conic passing through $\{x_4,...,x_8\}$ exists.
\begin{itemize}

\item If $x_{11}\in 2l-\sum_{i=4}^{8}E_i$, we can consider the exact sequence
$$0\rightarrow \mathcal{O}_{\overline{X'}}(4l-\sum_{i=4}^{8}E_i-2\sum_{i=9}^{10}E_i)\rightarrow\mathcal{O}_{\overline{X'}}(6l-2\sum_{i=4}^{10}E_i-E_{11})\rightarrow$$
\begin{equation}\label{P2 6}
  \rightarrow \mathcal{O}_{2l-\sum_{i=4}^{8}E_i-E_{11}}(6l-2\sum_{i=4}^{10}E_i-E_{11})\rightarrow 0.
\end{equation}

\noindent From the exact sequence (\ref{P2 3bisbis}), we know that $h^0(\mathcal{O}_{\overline{X'}}(4l-\sum_{i=4}^{8}E_i-2\sum_{i=9}^{10}E_i))=4$ and $h^1(\mathcal{O}_{\overline{X'}}(4l-\sum_{i=4}^{8}E_i-2\sum_{i=9}^{10}E_i))=0$.

\noindent Since $h^0(\mathcal{O}_{2l-\sum_{i=4}^{8}E_i-E_{11}}(6l-2\sum_{i=4}^{10}E_i-E_{11}))=2$ by Riemann-Roch's Theorem, then $h^0(\mathcal{O}_{\overline{X'}}(6l-2\sum_{i=4}^{10}E_i-E_{11}))=6$ and $h^1(\mathcal{O}_{\overline{X'}}(6l-2\sum_{i=4}^{10}E_i-E_{11}))=0$ from the exact sequence (\ref{P2 6}).

\item If $x_{11}\notin 2l-\sum_{i=4}^{8}E_i$, then we consider the following

\bigskip
$$0\rightarrow\mathcal{O}_{\overline{X'}}(4l-\sum_{i=4}^{8}E_i-2E_9-2E_{10}-E_{11})\rightarrow\mathcal{O}_{\overline{X'}}(6l-2\sum_{i=4}^{10}E_i-E_{11})\rightarrow
$$
\begin{equation}\label{P2 7}
 \rightarrow\mathcal{O}_{2l-\sum_{i=4}^{8}E_i}(6l-2\sum_{i=4}^{10}E_i-E_{11})\rightarrow0.
                                                              \end{equation}

\begin{itemize}
 \item[$\blacklozenge$] If $x_{11}\in l-E_9-E_{10}$, we consider
\bigskip
 $$
 0\rightarrow \mathcal{O}_{\overline{X'}}(3l-\sum_{i=4}^{10}E_i)\rightarrow\mathcal{O}_{\overline{X'}}(4l-\sum_{i=4}^{8}E_i-2\sum_{i=9}^{10}E_i-E_{11})\rightarrow$$
\begin{equation}\label{P2 8}
  \rightarrow\mathcal{O}_{l-E_{9}-E_{10}-E_{11}}(4l-\sum_{i=4}^{8}E_i-2\sum_{i=9}^{10}E_i-E_{11})\rightarrow0.
\end{equation}

 It is obvious that $h^0(\mathcal{O}_{\overline{X'}}(3l-\sum_{i=4}^{10}E_i))=3$ and $h^0(\mathcal{O}_{l-E_{9}-E_{10}-E_{11}}(4l-\sum_{i=4}^{8}E_i-2\sum_{i=9}^{10}E_i-E_{11}))=0$. From the exact sequence (\ref{P2 8}), we obtain that $h^0(\mathcal{O}_{\overline{X'}}(4l-\sum_{i=4}^{8}E_i-2\sum_{i=9}^{10}E_i-E_{11}))=3$.

 \item[$\blacklozenge$] If $x_{11}\notin l-E_{9}-E_{10}$, we can consider the following exact sequences

\bigskip
$$0\rightarrow\mathcal{O}_{\overline{X'}}(3l-\sum_{i=4}^{11}E_i)\rightarrow\mathcal{O}_{\overline{X'}}(4l-\sum_{i=4}^{8}E_i-2\sum_{i=9}^{10}E_i-E_{11})\rightarrow$$
\begin{equation}\label{P2 9}
 \rightarrow\mathcal{O}_{l-E_9-E_{10}}(4l-\sum_{i=4}^{8}E_i-2\sum_{i=9}^{10}E_i-E_{11})\rightarrow0;
                                                              \end{equation}

\bigskip
$$0\rightarrow\mathcal{O}_{\overline{X'}}(l-\sum_{i=9}^{11}E_i)\rightarrow\mathcal{O}_{\overline{X'}}(3l-\sum_{i=4}^{11}E_i)\rightarrow$$

\begin{equation}\label{P2 10}
\rightarrow\mathcal{O}_{2l-\sum_{i=4}^{8}E_i}(3l-\sum_{i=4}^{11}E_i)\rightarrow0.
                                                              \end{equation}

 \noindent By assumption we have that $h^0(\mathcal{O}_{\overline{X'}}(l-\sum_{i=9}^{11}
 E_i))=0$. Because \\\ $h^0(\mathcal{O}_{2l-\sum_{i=4}^{8}E_i}(3l-\sum_{i=4}^{11}E_i))=2$, then, from the exact sequence (\ref{P2 10}) we conclude that $h^0(\mathcal{O}_{\overline{X'}}(3l-\sum_{i=4}^{11}E_i))\leq2$. Since $h^0(\mathcal{O}_{\overline{X'}}(3l-\sum_{i=4}^{11}E_i))\geq\binom{5}{2}-8=2$, then equality holds.

 \par\bigskip\noindent Moreover $h^0(\mathcal{O}_{\overline{X'}}(4l-\sum_{i=4}^{8}E_i-2\sum_{i=9}^{10}E_i-E_{11}))\geq\binom{6}{2}-6-3-3=3$. Since $h^0(\mathcal{O}_{l-E_9-E_{10}}(4l-\sum_{i=4}^{8}E_i-2\sum_{i=9}^{10}E_i-E_{11}))=1$, then, from the exact sequence (\ref{P2 9}), we have that $h^0(\mathcal{O}_{\overline{X'}}(4l-\sum_{i=4}^{8}E_i-2\sum_{i=9}^{10}E_i-E_{11}))\leq3$, so equality holds.
 \end{itemize}

\par\bigskip\noindent In both previous cases, we have found $h^0(\mathcal{O}_{\overline{X'}}(4l-\sum_{i=4}^{8}E_i-2\sum_{i=9}^{10}E_i-E_{11}))=3$. Since it is the expected dimension, then $h^1(\mathcal{O}_{\overline{X'}}(4l-\sum_{i=4}^{8}E_i-2\sum_{i=9}^{10}E_i-E_{11}))=0$.

\noindent Using the Riemann-Roch Theorem, we have that \linebreak$h^0(\mathcal{O}_{2l-\sum_{i=4}^{8}E_i}(6l-2\sum_{i=4}^{10}E_i-E_{11}))=3$. So we obtain that $h^0(\mathcal{O}_{\overline{X'}}(6l-2\sum_{i=4}^{10}E_i-E_{11}))=6$ and $h^1(\mathcal{O}_{\overline{X'}}(6l-2\sum_{i=4}^{10}E_i-E_{11}))=0$ from the exact sequence (\ref{P2 7}).

\end{itemize}

\par\bigskip\noindent In all two cases we have that $h^0(\mathcal{O}_{\overline{X'}}(6l-2\sum_{i=4}^{10}E_i-E_{11}))=6$.

\par\bigskip\noindent With the same techniques as before, from the exact sequence (\ref{P2 5}), we have that $h^0(\mathcal{O}_{\overline{X'}}(12l-2\sum_{i=1}^{3}E_i-4\sum_{i=4}^{10}E_i-E_{11}))=11$ and $h^1(\mathcal{O}_{\overline{X'}}(12l-2\sum_{i=1}^{3}E_i-4\sum_{i=4}^{10}E_i-E_{11}))=0$ and finally, from the exact sequence (\ref{P2 4}), we obtain that $h^0(\mathcal{O}_{\overline{X'}}(18l-4\sum_{i=1}^{3}E_i-6\sum_{i=4}^{10}E_i-E_{11}))=12$.
\noindent Then we can conclude that $L'$ is base-point free. \end{proof}

\par\bigskip\noindent By Bertini's Theorem, since $L'$ is base-point free, then the generic $C'\in L'$ is smooth.

\par\bigskip\noindent $\mathrm{CLAIM \hspace{0.2cm}5:}$ The linear system $L'$ defines an embedding outside the contracted curve $J$.

\begin{proof} It is sufficient to show that $\dim\overline{L'}=\dim (L')-2$, where either $\overline{L'}=|18l-4\sum_{i=1}^{3}E_i-6\sum_{i=4}^{10}E_i-E_{11}-E_{12}|$, for $E_{11}$ and $E_{12}$ the exceptional divisors associated with any two distinct points $x_{11}$ and $x_{12}$ not belonging to $J$, or $\overline{L'}=|18l-4\sum_{i=1}^{3}E_i-6\sum_{i=4}^{10}E_i-E_{11}-2E_{12}|$, for $E_{11}$ and $E_{12}$ the exceptional divisors associated with any two points $x_{11}$ and $x_{12}$ infinitely near not belonging to $J$.

 \bigskip\noindent We have already proved that $\dim(L')=12$. Moreover we have that $\dim(\overline{L'})\geq\binom{20}{2}-3\frac{4\cdot 5}{2}-7\frac{6\cdot 7}{2}-2-1=10$. If $\overline{C'}$ is a general curve in $\overline{L'}$ and $\overline{J}$ is the strict transform of $J$ on $\overline{X'}$, then $\overline{C'}\cdot \overline{J}=0$ since we choose $x_{11}$ and $x_{12}$ not belonging to $J$.

 \bigskip\noindent We will show that $L'$ defines an embedding outside the contracted curve $J$ assuming $x_{11}$ and $x_{12}$ distinct. The proof is similar for $x_{11}$ and $x_{12}$ infinitely near.

\noindent We can consider the following exact sequences:

 $$ 0\rightarrow\mathcal{O}_{\overline{X'}}(12l-2\sum_{i=1}^{3}E_i-4\sum_{i=4}^{10}E_i-E_{11}-E_{12})\rightarrow$$
 \begin{equation}\label{flu 1}\rightarrow\mathcal{O}_{\overline{X'}}(18l-4\sum_{i=1}^{3}E_i-6\sum_{i=4}^{10}E_i-E_{11}-E_{12})\rightarrow\mathcal{O}_{\mathbb{P}^1}\rightarrow0;
\end{equation}
$$0\rightarrow\mathcal{O}_{\overline{X'}}(6l-2\sum_{i=4}^{10}E_i-E_{11}-E_{12})\rightarrow\mathcal{O}_{\overline{X'}}(12l-2\sum_{i=1}^{3}E_i-4\sum_{i=4}^{10}E_i-E_{11}-E_{12})\rightarrow$$
\begin{equation}\label{flu 2}
 \rightarrow\mathcal{O}_{\overline{J}}(12l-2\sum_{i=1}^{3}E_i-4\sum_{i=4}^{10}E_i-E_{11}-E_{12})\rightarrow0
\end{equation}
$$0\rightarrow\mathcal{O}_{\overline{X'}}(2\sum_{i=1}^{3}E_i-E_{11}-E_{12})\rightarrow\mathcal{O}_{\overline{X'}}(6l-2\sum_{i=4}^{10}E_i-E_{11}-E_{12})\rightarrow$$
\begin{equation}\label{flu 3}
 \rightarrow\mathcal{O}_{\overline{J}}(6l-2\sum_{i=4}^{10}E_i-E_{11}-E_{12})\rightarrow0.
\end{equation}

\noindent It is clear that $h^0(\mathcal{O}_{\overline{X'}}(2\sum_{i=1}^{3}E_i-E_{11}-E_{12}))=0$. Again, by Serre Duality, we have that $h^2(\mathcal{O}_{\overline{X'}}(2\sum_{i=1}^{3}E_i-E_{11}-E_{12}))=0$.
Using the Riemann-Roch Theorem, we obtain that $-h^1(2\sum_{i=1}^{3}E_i-E_{11}-E_{12}))=\frac{1}{2}(2\sum_{i=1}^{3}E_i-E_{11}-E_{12})\cdot(3l+\sum_{i=1}^{3}E_i-\sum_{i=4}^{10}E_i-2E_{11}-2E_{12})+1=-4$.

\noindent Since $\overline{J}\cdot (6l-2\sum_{i=4}^{10}E_i-E_{11}-E_{12})=8$, then $h^0(\mathcal{O}_{\overline{J}}(6l-2\sum_{i=4}^{10}E_i-E_{11}-E_{12}))=9$. Moreover $h^0(\mathcal{O}_{\overline{X'}}(6l-2\sum_{i=4}^{10}E_i-E_{11}-E_{12}))\geq\binom{8}{2}-7\cdot 3-2=5$. To show that equality holds, it is sufficient to prove that $h^1(\mathcal{O}_{\overline{X'}}(6l-2\sum_{i=4}^{10}E_i-E_{11}-E_{12}))=0$.

\noindent We observe that curves of the type $(3l-\sum_{i=4}^{12}E_i)+F$, with $F\in|3l-\sum_{i=4}^{10}E_i|$ and $3l-\sum_{i=4}^{12}E_i$ fixed part, are contained in $|6l-2\sum_{i=4}^{10}E_i-E_{11}-E_{12}|$.

\noindent We have that $\dim|6l-2\sum_{i=4}^{10}E_i-E_{11}-E_{12}|\geq5$ while $\dim |3l-\sum_{i=4}^{10}E_i|=\binom{5}{2}-7=\frac{5\cdot 4}{2}-7=3$, so the reducible curves of the type $(3l-\sum_{i=4}^{12}E_i)+F$ do not fill up all the linear system of the curves of degree $6$. As consequence of Bertini's Theorem (see \cite{Ak}, pag. $1$), the generic curve $D$ of $|6l-2\sum_{i=4}^{10}E_i-E_{11}-E_{12}|$ is irreducible (indeed the curves of the linear system $|6l-2\sum_{i=4}^{10}E_i-E_{11}-E_{12}|$ with fixed part $3l-\sum_{i=4}^{12}E_i$ define a sublinear system and moreover the sublinear system is not composed by a pencil, even more so the linear system $|6l-2\sum_{i=4}^{10}E_i-E_{11}-E_{12}|$).

\noindent Let us consider the exact sequence
\begin{equation}\label{flu 55}
0\rightarrow \mathcal{O}_{\overline{X'}}\rightarrow \mathcal{O}_{\overline{X'}}(D)\rightarrow \mathcal{O}_{D}(D)\rightarrow 0. \end{equation}

\noindent Since $D^2=36-28-2=6$ and $p_a(D)=\frac{5\cdot 4}{2}-7=3$, then $h^1(\mathcal{O}_{D}(D))=0$ (see \cite{H}, Example $IV.1.3.4$). Because $h^1(\mathcal{O}_{\overline{X'}})=0$ by definition, then \linebreak$h^1(\mathcal{O}_{\overline{X'}}(6l-2\sum_{i=4}^{10}E_i-E_{11}-E_{12}))=0$ from the exact sequence (\ref{flu 55}). Consequently $h^0(\mathcal{O}_{\overline{X'}}(6l-2\sum_{i=4}^{10}E_i-E_{11}-E_{12}))=5$ from the exact sequence (\ref{flu 3}).

\noindent Since $\overline{J}$ is rational, then $h^0(\mathcal{O}_{\overline{J}}(12l-2\sum_{i=1}^{3}E_i-4\sum_{i=4}^{10}E_i-E_{11}-E_{12}))=5$, so, from the exact sequence (\ref{flu 2}), we have that $h^0(\mathcal{O}_{\overline{X'}}(12l-2\sum_{i=1}^{3}E_i-4\sum_{i=4}^{10}E_i-E_{11}-E_{12}))=10$ and $h^1(\mathcal{O}_{\overline{X'}}(12l-2\sum_{i=1}^{3}E_i-4\sum_{i=4}^{10}E_i-E_{11}-E_{12}))=0.$

\noindent Finally, from the exact sequence (\ref{flu 1}), we obtain that $h^0(\mathcal{O}_{\overline{X'}}(18l-4\sum_{i=1}^{3}E_i-6\sum_{i=4}^{10}E_i-E_{11}-E_{12}))=11$. The claim is proved.
\end{proof}
\par\bigskip\noindent We have found a new example of rational surface $X\subset\mathbb{P}^{12}$ of degree $\deg(X)=\frac{C'^2}{\deg \phi_{L'}}=24$ with Prym-canonical hyperplane sections
 and only one singularity, a quartic rational singularity.
 \end{exa}

\end{document}